\documentclass[12pt]{amsart}
\usepackage{epsfig,color}
\usepackage{amssymb} 

\makeatother

\theoremstyle{definition}

\def\fnum{equation}
\newtheorem{Thm}[\fnum]{Theorem}
\newtheorem{Cor}[\fnum]{Corollary}

\newtheorem{Lem}[\fnum]{Lemma}

\newtheorem{Def}[\fnum]{Definition}

\newtheorem{Pro}[\fnum]{Proposition}

\numberwithin{equation}{section}
\newcommand{\Ric}{{\text{Ric}}}

\newcommand{\Tr}{{\text{Tr}}}

\newcommand{\Hess}{{\text {Hess}}}

\def\RR{{\bold R}}

\def\SS{{\bold S}}

\newcommand{\dv}{{\text {div}}}
\newcommand{\e}{{\text {e}}}

\newcommand{\cC}{{\mathcal{C}}}

\newcommand{\cg}{{\mathfrak{g}}}

\newcommand{\cF}{{\mathcal{F}}}
\newcommand{\cL}{{\mathcal{L}}}

\newcommand{\cP}{{\mathcal{P}}}

\newcommand{\cS}{{\mathcal{S}}}

\newcommand{\cR}{{\mathcal{R}}}
\newcommand{\cV}{{\mathcal{V}}}

\newcommand{\eqr}[1]{(\ref{#1})}

\def\bH{{\bold H}}
\def\bN{{\bold N}}

\title[Regularity of  elliptic and parabolic systems]{Regularity of elliptic and parabolic systems}
\author[]{Tobias Holck Colding}%
\address{MIT, Dept. of Math.\\
77 Massachusetts Avenue, Cambridge, MA 02139-4307.}
\author[]{William P. Minicozzi II}%
 
\thanks{The  authors
were partially supported by NSF Grants DMS 1812142 and DMS 1707270.}

\email{colding@math.mit.edu  and minicozz@math.mit.edu}

\begin{document}

\maketitle

\begin{abstract}
We show   uniqueness of cylindrical blowups for mean curvature flow in all dimension and all codimension.  Cylindrical singularities are known to be the most important; they are the most prevalent in any codimension.  
Mean curvature flow in higher codimension is a nonlinear parabolic system where many of the methods used for hypersurfaces do not apply and uniqueness of cylindrical blowups remained a major open problem.  
Our results imply regularity of the singular set for the system.  
\end{abstract}

\section{Introduction}

We show that  blowups (tangent flows) are unique at each cylindrical singularity of  a mean curvature flow (MCF)  in arbitrary codimension.  Uniqueness means that  blowups are  independent of the choice of rescalings as one magnifies around the singularity.  This  is one of the most fundamental issues about singularities, it has been studied in many geometric problems,  and it has major implications for regularity.  

There had been two general methods for proving uniqueness, going back to Allard-Almgren \cite{AA} and   Simon \cite{S}, and most uniqueness results use one of these.  In both \cite{AA} and \cite{S}, the singularity  is compact 
 and this is essential in the argument.  However, for MCF, as well as many other important problems, the most important singularities are non-compact.  If a flow starting at a closed  submanifold has a non-compact singularity, the evolving MCF is never a graph over the entire non-compact limit.  This makes it difficult to linearize the problem, or even to make sense of a neighborhood of the limit,  and creates other serious issues.
   A singular point is cylindrical if at least one tangent flow is a multiplicity one cylinder $\SS^k_{\sqrt{2k}} \times \RR^{n-k}$.    These are the most important singularities and
  neither of the two general methods in \cite{AA} and \cite{S} applies in this setting.  Uniqueness for cylinders was a major open problem for decades.   For hypersurfaces, this problem was solved in \cite{CM3}, but the approach relied heavily on it being a hypersurface and the general case for systems remained a major open problem.  
In this  paper, we prove the general uniqueness theorem in arbitrary  codimension.

\begin{Thm}	\label{t:main}
Let $M_t^n$ be a MCF in $\RR^N$.  At each cylindrical singular point the tangent flow is unique.  That is, any other tangent flow is also a cylinder with the same $\RR^k$ factor that points in the same direction.
\end{Thm}

Cylinders are the most significant singularities.  They are the most prevalent in any codimension.  By White's parabolic Almgren-Federer dimension reduction, \cite{W3},  the singular set is stratified into subsets whose Hausdorff dimension is the dimension of translation invariance of the blowup.  Thus, for multiplicity one flows, the most prevalent    singularities are points where the blowup is $\gamma\times \RR^{n-1}$ for a shrinking curve $\gamma$.
By uniqueness of solutions to ODEs,    $\gamma$ is a closed planar  Abresch-Langer curve that is a round circle if embedded or stable, \cite{CM1}, or with low entropy $\lambda (\gamma) < 2$.

  The  new approach   in
  \cite{CM3} to prove uniqueness   in codimension one was divided into two main steps -  extension and improvement - applied iteratively.
  The iteration scheme in \cite{CM3} showed that 
   the flow was converging to the limit on ever larger sets as time evolves,
  giving a way to deal with non-compact singularities.  In principle, this scheme can apply to a  variety of problems with non-compact limits, as long as one can establish extension and improvement  estimates.
  The extension was quite general, relying largely on monotonicity and Allard/Brakke methods that hold in all codimension.  The improvement step was much more complicated,  quite delicate, and relied on a quantitative form of the classification of hypersurface shrinkers with positive mean curvature.  This classification is a type of Bernstein theorem for shrinkers.   No such classification holds in higher codimension, just as there is  no Bernstein theorem for minimal submanifolds in higher codimension.  Entirely new ideas are needed to attack uniqueness for cylindrical blowups in higher codimension.  We do this here.  Some of the ideas developed  here play an important role in Ricci flow, \cite{CM9}.

Huisken's seminal paper \cite{H1} gave the first uniqueness of blowups for MCF, showing uniqueness of spherical singularities for hypersurfaces.     Uniqueness for spheres in higher codimension came much later in the work of
 Andrews-Baker, \cite{AB}.   
 Schulze   proved  uniqueness for closed singularities in all codimension, \cite{Sc}, using the method of \cite{S}, which applies to systems if the singularity is compact.   The first general result in the non-compact setting was \cite{CM3} (cf. \cite{GK}
when the flow is a  graph over the entire cylinder);  this relied heavily on properties of hypersurfaces.     In a   recent paper, Chodosh and Schulze, \cite{ChS}, 
  proved uniqueness for asymptotically conical shrinkers.

   \subsection{Key difficulties in higher codimension}

It is well-known that   higher codimension for MCF creates many new challenges as it does for minimal varieties, \cite{A}, \cite{D}.   There are many phenomena for hypersurfaces that simply do not extend to higher codimension. 
New challenges include the vector-valued second fundamental form,   
a   new curvature term $P$ that   causes major difficulties, and the lack of a maximum principle.
Moreover, the second fundamental form evolves by a reaction-diffusion system in which the reaction terms are  very complicated, while  they are quite easily understood for hypersurfaces.

The estimates that we use to prove  uniqueness  involve both a Simons type identity and  the linearized operator which involves  terms that arise in the second variation for minimal submanifolds.  These terms do not have a sign in higher codimension, which is one of the reasons why there is no higher codimension stable Bernstein theorem in the minimal case.  In our case, these terms give a new curvature term $P$ defined in \eqr {e:definePhere}
 that vanishes for hypersurfaces but does not have a sign in general.
Controlling  $P$ is a crucial and  delicate point that did not arise for hypersurfaces,  cannot be handled at the linear level, and can only be controlled in certain directions.
 Dealing with  $P$ is one of the key points in the paper.   
 
In addition to understanding $P$, there are several other   interesting new phenomena in higher codimension.   For instance,  the shrinker equation leads most naturally to controlling the second fundamental form in the direction of the mean curvature vector.  For hypersurfaces, there is only one normal direction so this is sufficient.  In higher codimension, we have much less control which presents new challenges that
are dealt with in Section \ref{s:class}.
 
  \subsection{Applications}
  
  Uniqueness of blowups has strong implications for the   structure of flows and the regularity of the  set $\cS$ where the flow becomes singular.   In codimension one, uniqueness was a crucial ingredient in sharp regularity of $\cS$, \cite{CM7}.    
Combining Theorem \ref{t:main} with the ideas introduced in \cite{CM7}, we extend this to higher codimension\footnote{The main theorem of \cite{CM7} states the submanifolds are Lipschitz, but, as noted in \cite{CM6}, the proof shows that they are $C^1$ submanifolds.}:   

\begin{Cor}   \label{c:cmain}
 Let $M_t^n \subset \RR^N$ be a MCF with only cylindrical singularities, then the space-time singular set $\cS$ satisfies:
\begin{itemize}
\item $\cS$ is contained in finitely many (compact) embedded $C^1$ submanifolds each of dimension at most $(n - 1)$ together with a set of dimension at most $(n - 2)$.
\end{itemize}
\end{Cor}

  The main ingredients in \cite{CM7} were monotonicity, which holds in all codimension, and a quantitative form of uniqueness.   This paper proves the desired uniqueness    in all codimension and, thus, the argument in \cite{CM7} gives Corollary \ref{c:cmain}.
    For simplicity,   Corollary \ref{c:cmain} is stated just for the top-dimensional part    of $\cS$, but much more is true; cf. \cite{CM7}. 

 In codimension one uniqueness was also a crucial ingredient in 
  optimal regularity for the degenerate elliptic arrival time equation,
 \cite{CM2},  \cite{CM5}, \cite{CM6}.   
   In higher codimension, the level set equation becomes a system;  \cite{AmS}, \cite{B},  \cite{CnGG},  \cite{ES}, \cite{OF},   \cite{OS}, \cite{Wa}.   
Uniqueness in higher codimension   should have similar applications to  this degenerate system that is much less understood; see, e.g.,   page $465$ in  \cite{OF}.
 Uniqueness also played an  important role in  classifying certain ancient flows; see \cite{BC}, \cite{ChHH}, \cite{ChHHW} and references therein. 
 
 \subsection{Overview of the key steps}		\label{ss:overview}
 
 Rescaled MCF  is the negative gradient flow for
the Gaussian area $F$
\begin{align}	\label{e:gaussianF}
	F(\Sigma) =  \left( 4 \pi \right)^{ - \frac{n}{2}} \, \int_{\Sigma} \e^{ - \frac{|x|^2}{4} } \, .
\end{align}
Up to reparameterizations, rescaled MCF  is equivalent to continuously rescaling a MCF.
 Uniqueness of blowups for MCF is then equivalent to uniqueness of limits for rescaled MCF.   Suppose, thus, that $\Sigma_t$ is a rescaled MCF and $t_i \to \infty$ is a sequence where $\Sigma_{t_i}$ converges to a cylinder.
 The point is   to show that $\Sigma_t$ converges to the cylinder at some definite rate.  This is delicate and any rate must be  slow due to the presence of non-integrable Jacobi fields.  A rate would follow  from   an inequality for $F$ of the form
 \begin{align}	\label{e:LjF}
 	\left|F(\Sigma_t) - \lim_{t\to \infty} F(\Sigma_t) \right| \lesssim  \| \nabla F \|_{L^2}^{1+ \tau} \, , 
\end{align}
where $\tau > 0$ and the inequality need only hold near a cylinder.
 The proof of this  estimate  is given in Proposition \ref{t:discre}.    We will explain the   ideas that go into the proof next, suppressing  technical points   to give the broad picture.{\footnote{For this discussion, we ignore  constants, $\log$ terms, and  differences between $L^2$ norms and $W^{1,2}$ norms, etc.  These  technical points are lower order and are addressed by giving up a little in the exponents.}}
 
One of the key difficulties will be that $\Sigma_t$ will not be a graph over the entire cylinder.  We will use two measures of  closeness to a cylinder: the radius $r$ of the ball it is graphical over and the Gaussian $L^2$ 
norm $\| U(\cdot , t) \|_{L^2}$ of the graph.  Here $U(x,t)$ will be a normal vector field on the cylinder and we will assume that $U$ has compact support and the intersection of the flow with $B_r$ is contained in the graph of $U$.

It will be useful to consider both  $\Sigma_t$ and the graph $\Sigma_U$ of $U$; these agree in  $B_r$ but differ outside.  It is easier to work with $\Sigma_U$ and, in fact, Proposition \ref{p:gradLoj} will (roughly) show that
\begin{align}	\label{e:6p5}
	|F(\Sigma_U) -   \lim_{t\to \infty} F(\Sigma_t)| \lesssim  \| \nabla_{\Sigma_U} F \|_{L^2} \, \| U \|_{L^2} + \| U \|_{L^2}^3 \, .
\end{align}
Since $\Sigma_U$ and $\Sigma_t$ agree in  $B_r$, combining \eqr{e:6p5} with the triangle inequality gives
\begin{align}	\label{e:6p5a}
	|F(\Sigma_t) -  \lim_{t\to \infty} F(\Sigma_t)| \lesssim \left(  \| \nabla  F \|_{L^2} + \e^{ - \frac{r^2}{8} } \right) \, \| U \|_{L^2} + \| U \|_{L^2}^3 + \e^{ - \frac{r^2}{4}}  \, .
\end{align}
The crucial estimates on $r$ and $\| U \|_{L^2}$ are then given in 
 Theorem \ref{t:shrinkerscale}, which roughly shows that there is some $\beta > 0$ so that 
 \begin{align}
 	\e^{ - \frac{r^2}{4} } &\lesssim  \| \nabla F \|_{L^2}^{1+ \beta} \, , \label{e:t72a} \\
	\| U \|_{L^2}^2 &\lesssim  \| \nabla F \|_{L^2} \, .  \label{e:t72b}
 \end{align}
 The inequality \eqr{e:LjF} follows by 
  using \eqr{e:t72a} and \eqr{e:t72b} in \eqr{e:6p5a}.
   
 To explain the ideas behind \eqr{e:t72a} and \eqr{e:t72b}, we need
     a tensor $\tau$ that is parallel on cylinders.  
    Proposition \ref{p:Lojae1A} roughly gives that
\begin{align}
	\| U \|_{L^2}^2  \lesssim \| \nabla F \|_{L^2} + \| \nabla \tau \|_{L^2}^2 \, ,
\end{align}
so we must bound   $\nabla \tau$.  This is where the  quantity $P$ comes in as
  Theorem \ref{c:kappa}  gives  
 \begin{align}	\label{e:from2q2}
 	\| \nabla \tau \|_{L^2}^2  \lesssim \| P \|_{L^1} + \| \nabla F \|_{L^2} + \e^{ - \frac{r^2}{4} } \, .
 \end{align}
 The estimate comes from integrating a differential inequality for $\tau$ and the exponential term comes from cutting off at scale $r$.
 The quantity $P$ vanishes for hypersurfaces, but not in higher codimension.  However,
  Proposition \ref{l:PLone} gives an $L^1$ bound on $P $  
\begin{align}	 \label{e:fromPLone}
	\| P \|_{L^1}  \lesssim  \| U \|_{L^2}^3 + \| U \|_{L^2} \, \| \nabla F \|_{L^2}  + \| \nabla F \|_{L^2}^2 
	 \lesssim  \| U \|_{L^2}^{2(1+ \nu)} +  \| \nabla F  \|_{L^2}^{2(1-\nu)} \, ,
\end{align}
where the second inequality holds for any $\nu > 0$ by Young's inequality.
Combining these estimates gives that
 \begin{align}	 
 	\| U \|_{L^2}^2  \lesssim   \| U \|_{L^2}^{2(1+ \nu)} + \| \nabla F \|_{L^2} + \e^{ - \frac{r^2}{4} } \, .
 \end{align}
This allows us to absorb the first term on the right to get
 \begin{align}	 	\label{e:herewego}
 	\| U \|_{L^2}^2  \lesssim    \| \nabla F \|_{L^2} + \e^{ - \frac{r^2}{4} } \, .
 \end{align}
To complete the proof of \eqr{e:t72a} and \eqr{e:t72b}, we combine \eqr{e:herewego} with the extension, which holds in all codimension, from \cite{CM3}.  Roughly, this extension says that if $\Sigma_t$ is a small graph over the cylinder on a scale $s$ with $\| \nabla F \|_{L^2} \lesssim \e^{ - \frac{s^2}{4}}$, then it remains a graph (with worse bounds) out to scale $(1+ \alpha)s$.   On the other hand, \eqr{e:herewego} gives improved estimates for the graph on any scale up to this.  When we iterate these two steps, we keep moving to larger scales and eventually get \eqr{e:t72a}.  Using \eqr{e:t72a} in \eqr{e:herewego} then gives \eqr{e:t72b}.
 
Finally,    it was crucial in proving \eqr{e:herewego} that the exponent on $\| U \|_{L^2}$ in \eqr{e:fromPLone} was greater than two.  The estimate  \eqr{e:fromPLone} will be proven by Taylor expanding $P$ around the cylinder.  We will see that $P$ and its gradient vanish at the cylinder, which formally gives 
$
	\| P \|_{L^1}  \lesssim \| U \|_{L^2}^2 \, .
$
However, this would not be good enough.  A higher power would follow if the  Hessian term in the Taylor expansion of $P$ also vanishes, but, unfortunately, this is not true.  However,   the Hessian of $P$ is better in certain directions and this will be just enough to get   \eqr{e:fromPLone}.

    \section{Submanifolds in higher codimension}
    
Let  $\Sigma^n \subset \RR^N$  be an $n$-dimensional submanifold with second fundamental form $A(X,Y) = \nabla_X^{\perp} Y$.  In contrast to   hypersurfaces, $A$ is now a vector-valued two-tensor, taking values in the normal bundle.   Given a normal vector field $V$, let $A^V$ be the real-valued symmetric two-tensor $A^V = \langle A , V \rangle$.  
Let $E_i$ be an orthonormal frame for the tangent space of $\Sigma$ 
 and $\bH = - A (E_i , E_i) = - A_{ii}$  the  mean curvature vector.       Let $\phi$ be the normal vector field
\begin{align}
	\phi = \frac{1}{2} \, x^{\perp} - \bH \, .
\end{align}
Rescaled MCF is given by $x_t = \phi$; static solutions, with $\phi = 0$, are shrinkers.
 When $\bH \ne 0$,  the principal normal is $\bN = \frac{\bH}{|\bH|}$ and  the tensor 
$
	\tau   \equiv \frac{A}{|\bH|}$.  Note that the trace of $\tau$ is $-\bN$.

Let $\e^{-f} = \e^{ - \frac{|x|^2}{4} }$ be the Gaussian weight.  The Gaussian $L^2$ inner product of functions $u$ and $v$ is  
$
	 \left( 4 \pi \right)^{ - \frac{n}{2}} \,  \int_{\Sigma} u v \, \e^{ - f}$ 
	and  $\| u \|_{L^2}$   denotes the Gaussian $L^2$ norm.
Following \cite{CM1}, the entropy $\lambda$ is the supremum of $F$ over all translations and dilations
 $\lambda (\Sigma) = \sup_{c,x_0} \, F (c\,\Sigma + x_0)$.
 
A tangent flow is the limit of a sequence of rescalings at a singularity. For instance, a tangent flow to $M_t$ at the origin in space-time is the limit of a sequence of rescaled flows $\frac{1}{\delta_i}\,M_{\delta_i^{2}\,t}$ where $\delta_i\to 0$.  A priori, different sequences $\delta_i$ could give different tangent flows. By a monotonicity formula of Huisken, \cite{H2}, and an argument of Ilmanen and White, \cite{I}, \cite{W1}, tangent flows are shrinkers, i.e., self-similar solutions of MCF that evolve by rescaling.

Define the  operators $\cL$ and $L$  (cf. \cite{CM1}) by  $\cL  =  \Delta   - \frac{1}{2} \, \nabla_{x^T}$ and 
\begin{align}
	L  &= \cL  + \frac{1}{2} + \sum_{k,\ell}  \,  \langle  \cdot   , A_{k \ell}  \rangle \, A_{ k \ell} \, .
\end{align}
The operator $\cL$ is defined on functions or, more generally, on tensors, while $L$ is defined  on normal vector fields and, more generally,
tensors with values in the normal bundle.
Both operators  are self-adjoint with respect to the Gaussian inner product.

One of the important ingredients in this paper is the next general
  Simons identity for $A$ and $\bH$ in higher codimension.  The terms involving $\phi$ drop out for a shrinker.

\begin{Pro}	\label{p:simonsG}
We have
\begin{align}	\label{e:simonsG1}
	\left(L \, A \right)_{ij} &=     A_{ij}     
	+     2\,  \sum_{k, \ell}   \,  \langle A_{j \ell}   , A_{ik}  \rangle \, A_{\ell k}    -    \sum_{m,\ell}   \,
	 \left\{ \langle A_{m \ell}  , A_{i \ell}  \rangle A_{j m}   + \langle A_{j \ell}   , A_{m \ell}  \rangle A_{im}  \right\}   \notag \\
	 &+  \Hess_{\phi} (E_i, E_j)  +  \sum_m A_{jm}^{\phi}  A_{im} 
          \, , \\
          L \, {\bf{H}} &= {\bf{H}} - \Delta \phi -  \sum_{i,m} A_{im}^{\phi}  A_{im} \,  . \label{e:simonsG2}
\end{align}
\end{Pro}

Proposition \ref{p:simonsG} will be proven below in Section \ref{s:genSimo}.  It is interesting to compare \eqr{e:simonsG1} with the much simpler  
 proposition $1.2$ in \cite{CM3} for hypersurfaces where $2\,  \sum_{k, \ell}   \,  \langle A_{j \ell}   , A_{ik}  \rangle \, A_{\ell k}    -    \sum_{m,\ell}   \,
	 \left\{ \langle A_{m \ell}  , A_{i \ell}  \rangle A_{j m}   + \langle A_{j \ell}   , A_{m \ell}  \rangle A_{im}  \right\} = 0$.  In particular, a hypersurface shrinker has $L \, H = H$ and $L \, A = A$, so the ratio $\frac{A}{H}$ has special significance.  
	 The failure of this in higher codimension is a crucial point in this paper.

 \section{An integral bound for $\nabla \frac{A}{|\bH|}$}	\label{s:genSimo}
 
  Throughout this section, $\Gamma^n \subset \RR^N$  is an $n$-dimensional submanifold.  The tensor $\tau =  \frac{A}{|\bH|}$ is parallel on cylinders and we will show that, roughly,   $|\nabla \tau |$ controls the distance to a cylinder (this will be established in the next section).  In this section, we will bound
  $\nabla \tau$ using that $\tau$ satisfies a drift equation.  This equation is much more  complicated in higher codimension where a new quantity $P$ comes in
\begin{align}	\label{e:definePhere}
	P & \equiv |A|^2 \,|A^{\bN}|^2-    
	2 |A^2|^2  
	   + \sum_{i,j,k,\ell} \left\{  2\,   \langle A_{j \ell}   , A_{ik}  \rangle  \, \langle A_{\ell k} , A_{ij} \rangle \,      -  \langle A_{ij} ,  A_{k\ell} \rangle^2  \right\}   \notag \\
	   &+ \frac{|A|^2}{4\,|\bH|^2} \, \left( \left| A^{\bN} (x^T , \cdot) \right|^2 - \left| A  (x^T , \cdot) \right|^2 \right) \, ,
\end{align}
where $A^2$ is the real-valued symmetric two-tensor $(A^2)_{ij} = \langle A_{ik} , A_{kj} \rangle$ and $x^T$ is the tangential projection of the position vector $x$.  The quantity $P$ causes major problems in higher codimension.

The main result of this section is an integral bound for $\nabla \tau$ in terms of the quantity $P$.
 
 \begin{Thm}	\label{c:kappa}
If  $\psi$ is a compactly supported function   and $\bH \ne 0$ on the support of $\psi$, then 
\begin{align}	\label{e:ckappa}
	\int  \psi^2 \,\left(  |\nabla \tau|^2 \, |\bH|^2 + 2P \right)  \,  \e^{-f} & \leq 4\, \int  |A|^2 \, |\nabla \psi |^2 \, \e^{-f}  \notag \\
	&+ C\, \int \left(|A| + \frac{|A|^2}{|\bH|} \right) \,\left( |A|^2 \, |\phi| +  \left| \Hess_{\phi} \right| \right)\, \psi^2 \, \e^{-f} \\
	&+ C \, \int \frac{|A|^2\, |\nabla^{\perp} \phi|}{|\bH|^2} \left( |A(x^T , \cdot)| + |\nabla^{\perp} \phi| \right) \, \psi^2 \, \e^{-f} \, . \notag
\end{align}
\end{Thm}

This estimate will be used to bound $\nabla \tau$ in terms of $\phi$, the Gaussian energy of the cutoff $\psi$, and the $L^1$ norm of $P$.  Bounding $\| P \|_{L^1}$ is a significant challenge that cannot be handled at the linear level;  this bound takes up sections
\ref{s:s5} and \ref{s:s6}.

\vskip1mm
The next lemma shows that $P$ vanishes in codimension one.

\begin{Lem}	\label{l:PzeroHyper}
If $\Gamma^n$ is contained in an affine $(n+1)$-plane $W \subset \RR^N$ and $\bH \ne 0$, then $P\equiv 0$.
\end{Lem}

 \begin{proof}
Since $\Gamma$ is codimension one, $A = A^{\bN} \, \bN$. It follows that $\left| A^{\bN} (x^T , \cdot) \right|^2 = \left| A  (x^T , \cdot) \right|^2 $ and
   $|A|^2 \, |\bH|^2 = |A^{\bH}|^2$.  Fix a point and choose the orthonormal tangent frame $E_i$  so that $A^{\bN}$ is diagonal with eigenvalues $\lambda_i$ (summing to $-H$). We have
 \begin{align}
	P &= |A|^4  -
	2 |A^2|^2 
	   + \sum_{i,j,k,\ell} \left\{  2\,   \langle A_{j \ell}   , A_{ik}  \rangle  \, \langle A_{\ell k} , A_{ij} \rangle \,      -  \langle A_{ij} ,  A_{k\ell} \rangle^2  \right\} \notag \\
	   &=  \left( \sum_{i} \lambda_i^2 \right)^2  -
	2 \sum_i \lambda_i^4    
	   +  2 \sum_{i} \lambda_i^4 - \sum_{i,k} \lambda_i^2  \lambda_k^2  = 0  \, .
\end{align}
 \end{proof}

\subsection{Proof of the general Simons identity}

In this subsection, we will prove the general Simons identity   Proposition \ref{p:simonsG}.
The next   lemma  computes  two covariant derivatives of the mean curvature $\bH$.  The Hessian of a normal vector field $V$ does not have to be symmetric.  By convention,  we take
\begin{align}
	\Hess_V (E_i , E_j) = \nabla_{E_j}^{\perp} \nabla_{E_i}^{\perp} V - \nabla_{ \nabla_{E_j}^T E_i}^{\perp} V \, .
\end{align}

\begin{Lem}	\label{l:gradbH}
We have $\nabla_{E_i}^{\perp} \bH = 
	- \frac{1}{2} A(x^T , E_i ) - \nabla_{E_i}^{\perp} \phi $ and
\begin{align}	\label{e:231}
	- \Hess_{\bf{H}}  (E_i, E_j) &=  \Hess_{\phi}  (E_i, E_j) + \frac{1}{2} \,  \left( \nabla_{x^T} A\right)_{ij}
	+ \frac{1}{2}  A_{ij}
	 + \frac{1}{2}  A_{ik} \,  A^{x^{\perp}}_{jk}
 \, .
\end{align}
\end{Lem}

\begin{proof}
Differentiating $\phi = \frac{1}{2} x^{\perp} - \bH$ gives
\begin{align}
	\nabla_{E_i}^{\perp} \bH = \frac{1}{2} \nabla^{\perp}_{E_i} (x-x^T) - \nabla_{E_i}^{\perp} \phi = 
	- \frac{1}{2} A(x^T , E_i ) - \nabla_{E_i}^{\perp} \phi  \, .
\end{align}

The equation \eqr{e:231} is tensorial, so we can work at a point and assume that   $\nabla_{E_i}^T E_j = 0$ at this point.
Using the definition of $\phi$, we can write 
the Hessian of $\bH$ as
\begin{align}
	- \Hess_{\bf{H}} (E_i, E_j) =  \Hess_{\phi}  (E_i, E_j) - \frac{1}{2} \, \nabla_{E_j}^{\perp} \nabla_{E_i}^{\perp} x^{\perp} \, .
\end{align}
We simplify the last term using $\nabla_{E_i} x = E_i$ and the Codazzi equations (and $\nabla_{E_j}^T E_i = 0$)
\begin{align}
	- \nabla_{E_j}^{\perp} \left(  \nabla_{E_i}^{\perp} x^{\perp} \right) &=  \nabla_{E_j}^{\perp} \left(  \nabla_{E_i}^{\perp} x^T \right) = 
	 \nabla_{E_j}^{\perp} \left(  A(E_i , x^T) \right)  \notag \\
	 &=   \left( \nabla_{E_j} A\right)  \left( E_i , x^T \right)  +    A(E_i , \nabla^T_{E_j} x^T)  =  \left( \nabla_{x^T} A\right)  \left( E_i , E_j \right)  +    A(E_i , \nabla^T_{E_j} x^T) 
	 \, .
	\end{align}
Next, we note that
\begin{align}
	\langle E_k  ,\nabla_{E_j}^T x^T \rangle &= \langle E_k , \nabla_{E_j} x^T \rangle = \langle E_k  , \nabla_{E_j} (x-x^{\perp}) \rangle = \delta_{jk} - 
	\langle E_k  , \nabla_{E_j} x^{\perp} \rangle \notag \\
	&= \delta_{jk} +
	\langle  \nabla_{E_j} E_k  ,  x^{\perp} \rangle  = \delta_{jk} +   A^{x^{\perp}} (E_j , E_k)   \, .
\end{align}
It follows that 
\begin{align}
	 \nabla^T_{E_j} x^T  = E_j + A^{x^{\perp}} (E_j , E_k)  \, E_k \, .
\end{align}
Putting this together gives \eqr{e:231}.
\end{proof}

\begin{Cor}	\label{c:nablabN}
If $\bH \ne 0$ and $\bN = \frac{\bH}{|\bH|}$, then 
\begin{align}
	|\bH|^2 \, |\nabla \bN |^2 &= |\nabla \bH|^2 - |\nabla |\bH||^2 = \frac{1}{4} \, \left| A(x^T , \cdot)  \right|^2 - \frac{1}{4} \, \left|  A^{\bN} (x^T , \cdot ) \right|^2 \notag \\
	&+ \Tr \, A^{ \nabla_{\cdot}^{\perp}\phi }(x^T , \cdot )     + \left|  \nabla^{\perp} \phi \right|^2 - \Tr \, A^{\bN} (x^T , \cdot ) \langle \nabla_{\cdot}^{\perp} \phi , \bN \rangle 
	- \left| \langle \nabla^{\perp} \phi , \bN \rangle \right|^2 \, .
\end{align}
\end{Cor}

\begin{proof}
The first equality follows from $\nabla \bH = |\bH| \, \nabla \bN + (\nabla |\bH|)\, \bN$ and $\langle \bN , \nabla \bN \rangle = 0$.
Let $E_i$ be an orthonormal frame.
By Lemma \ref{l:gradbH},
  $\nabla_{E_i}^{\perp} \bH = 
	- \frac{1}{2} A(x^T , E_i ) - \nabla_{E_i}^{\perp} \phi $.  It follows that
\begin{align}	\label{e:comoo1}
	\left| \nabla \bH \right|^2 = \left| \frac{1}{2} A(x^T , E_i ) + \nabla_{E_i}^{\perp} \phi \right|^2 = \frac{1}{4} \, \left| A(x^T , E_i )  \right|^2 
	+  A^{ \nabla_{E_i}^{\perp}\phi }(x^T , E_i )     + \left|  \nabla_{E_i}^{\perp} \phi \right|^2 \, .
\end{align}
Similarly, we have for each $i$ that
\begin{align}
	  |\bH|  \, \nabla_{E_i} |\bH|  = \frac{1}{2} \,  \nabla_{E_i} |\bH|^2 = \langle \nabla_{E_i} \bH , \bH \rangle = -  \frac{1}{2}  A^{\bH} (x^T , E_i )   - \langle \nabla_{E_i}^{\perp} \phi , \bH \rangle \, .
\end{align}
It follows that
\begin{align}
	 |\bH|^2 \, |\nabla |\bH||^2 = \frac{1}{4} \, \left|  A^{\bH} (x^T , E_i ) \right|^2 + A^{\bH} (x^T , E_i ) \langle \nabla_{E_i}^{\perp} \phi , \bH \rangle + \left| \langle \nabla_{E_i}^{\perp} \phi , \bH \rangle \right|^2  \, .
\end{align}
Combining this with  \eqr{e:comoo1} gives the second equality.
\end{proof}

\begin{proof}[Proof of Proposition \ref{p:simonsG}]
We will work at a point using an orthonormal frame $E_i$ for the tangent space  where $\nabla^T E_i $ vanishes at this point.
The starting point is the general Simons' identity for $A$ (see, for instance, ($23$) on page $368$ of \cite{AB} where they have the opposite sign convention on $\bH$)
\begin{align}	\label{e:genS}
	\left( \Delta_{\Sigma} A \right) (E_i , E_j) &= -  \sum_{k,m} \langle A_{ij} , A_{km} \rangle \, A_{km}  - \Hess_{\bf{H}} (E_j, E_i) 
	- \sum_m \langle {\bf{H}} , A_{im} \rangle A_{jm}
	  \\
	& +   \sum_{k,m} \left\{  2\, \langle A_{j m} , A_{ik} \rangle \, A_{k m} - \langle A_{km} , A_{ik} \rangle A_{jm} - \langle A_{jk} , A_{km} \rangle A_{im} \right\}   
          \, .  \notag 
       \end{align}
       By the Ricci equations for the curvature of the normal bundle (equation $(15)$ on page $367$ of \cite{AB}), the combination $-  \sum_{k,m} \langle A_{ij} , A_{km} \rangle \, A_{km}  - \Hess_{\bf{H}} (E_j, E_i) $ is symmetric in $i$ and $j$.
      Thus, using \eqr{e:231} in \eqr{e:genS} gives  
\begin{align}	 
	\left( \Delta_{\Sigma} A \right) (E_i , E_j) &=      2\, \langle A_{j m} , A_{ik} \rangle \, A_{k m} - \langle A_{km} , A_{ik} \rangle A_{jm} - \langle A_{jk} , A_{km} \rangle A_{im}     -   \langle A_{ij} , A_{km} \rangle \, A_{km}  \notag \\
	&-   A_{jm}^{\bH}  A_{im} + \Hess_{\phi}  (E_i, E_j) + \frac{1}{2} \,  \left( \nabla_{x^T} A\right)  \left( E_i , E_j \right)  
	+ \frac{1}{2}  A_{ij}   + \frac{1}{2}  A_{im} \,  A_{jm}^{x^{\perp}} 
	          \, .   
       \end{align}
       Bringing the $\nabla_{x^T}A$ term to the left side and combining
       the first and last terms on the second line gives
       \begin{align}	 
	\left( \cL  A \right) (E_i , E_j) &=      2\, \langle A_{j m} , A_{ik} \rangle \, A_{k m} - \langle A_{km} , A_{ik} \rangle A_{jm} - \langle A_{jk} , A_{km} \rangle A_{im}     -   \langle A_{ij} , A_{km} \rangle \, A_{km}  \notag \\
	&+   A_{jm}^{\phi}  A_{im} + \Hess_{\phi}  (E_i, E_j)  
	+ \frac{1}{2}  A_{ij}   
	          \, .   
       \end{align}
The first claim \eqr{e:simonsG1} follows from this since $L = \cL  + \frac{1}{2} + \sum_{k,\ell}  \,  \langle  \cdot   , A_{k \ell}  \rangle \, A_{ k \ell}$.  The second claim follows by taking the trace of \eqr{e:simonsG1}.
\end{proof}

\subsection{The integral estimate for   $\nabla \tau$}

In the next lemma, we will use a second drift Laplacian $\cL_{|\bH|^2}$ (see lemma $4.3$ in \cite{CIM}) given by
\begin{align}
	\cL_{|\bH|^2} = \cL + \nabla_{\nabla \log |\bH |^2} \, .
\end{align}
This operator is self-adjoint with respect to the weight $|\bH|^2 \, \e^{ - f}$.
The   lemma gives an identity for $\cL_{|\bH|^2} \, \tau$ involving the quantity $P$; this identity is the motivation for the definition of $P$.  In the special case where $\Gamma$ is a shrinker and, thus, $\phi \equiv 0$, 
\eqr{e:e328} below becomes
\begin{align}	
	|\bH| \,\langle A ,  \cL_{|\bH|^2} \, \tau \rangle 
	 &=   P  
	  \, .
\end{align}

 \begin{Lem}	\label{l:kappa}
If $\bH \ne 0$, then $\tau = \frac{A}{|\bH|}$ satisfies
\begin{align}	\label{e:e328}
	|\bH| \,\langle A ,  \cL_{|\bH|^2} \, \tau \rangle 
	 &=   P +     A_{jm}^{\phi}  \langle A_{im}, A_{ij} \rangle  +  \langle \Hess_{\phi} , A \rangle + \frac{|A|^2}{|\bH|} \left( \langle \Delta \phi , \bN \rangle
	 + A^{\phi}_{ij} A^{\bN}_{ij} \right) \notag \\
	 &+ \frac{|A|^2}{|\bH|^2} \, \left\{ \Tr \, A^{\bN} (x^T , \cdot ) \langle \nabla_{\cdot}^{\perp} \phi , \bN \rangle 
	+ \left| \langle \nabla^{\perp} \phi , \bN \rangle \right|^2 -   \Tr \, A^{ \nabla_{\cdot}^{\perp}\phi }(x^T , \cdot )    - \left|  \nabla^{\perp} \phi \right|^2   
	 \right\} 
	  \, .
\end{align}
\end{Lem}

\begin{proof}
The first claim in 
Proposition \ref{p:simonsG}
gives that
 \begin{align}	 
	\langle  \cL  A   , A \rangle &=      2\, \langle A_{j m} , A_{ik} \rangle \, \langle A_{k m} , A_{ij} \rangle - 2 \, |A^2|^2   -   \langle A_{ij} , A_{km} \rangle^2 \notag \\
	&+   A_{jm}^{\phi}  \langle A_{im}, A_{ij} \rangle  + \langle \Hess_{\phi} , A \rangle
	+ \frac{1}{2}  |A|^2 
	          \, .   
       \end{align}
The second claim in Proposition \ref{p:simonsG} gives  $\cL \, \bH =  \frac{1}{2} {\bf{H}} - \Delta \phi -  A_{im}^{\phi}  A_{im} - A^{\bH}_{km} \, A_{km}$ and, thus,
\begin{align}
	\frac{1}{2} \, \cL \, |\bH|^2 = |\nabla \bH |^2 + \langle \cL \bH , \bH \rangle = |\nabla \bH |^2 + \frac{1}{2} \, |\bH|^2
	- \langle \Delta \phi , \bH \rangle -  A_{im}^{\phi}  A_{im}^{\bH} - \left| A^{\bH} \right|^2 \, .
\end{align}
Since $|\bH| \, \cL \, |\bH|  = \frac{1}{2} \, \cL \, |\bH|^2  - |\nabla |\bH| |^2 $, 
it follows that
\begin{align}	\label{e:cLh}
	|\bH| \, \cL \, |\bH|  &=    |\nabla \bH |^2 - |\nabla |\bH| |^2 + \frac{1}{2} \, |\bH|^2
	- \langle \Delta \phi , \bH \rangle -  A_{im}^{\phi}  A_{im}^{\bH} - \left| A^{\bH} \right|^2 \, .  
\end{align}

The quotient rule (lemma $4.3$ in \cite{CIM}) gives that
\begin{align}
	|\bH|^2 \, \cL_{|\bH|^2} \tau = |\bH|^2 \, \cL_{|\bH|^2}  \frac{A}{|\bH|} =   |\bH| \, \cL \, A - A \, \cL \, |\bH| \, .
\end{align}
Taking the inner product with $A$ and using the above formulas for $\cL$ on $A$ and $|\bH|$, we get 
\begin{align}
	|\bH|  \,\langle A  ,  \cL_{|\bH|^2} \tau \rangle  &=        2\,  \langle A_{j m} , A_{ik} \rangle \, \langle A_{k m} , A_{ij} \rangle - 2\,  |A^2|^2   -   \langle A_{ij} , A_{km} \rangle^2 \notag \\
	&+    A_{jm}^{\phi}  \langle A_{im}, A_{ij} \rangle  + \langle \Hess_{\phi}  , A \rangle
	+ \frac{1}{2} \,  |A|^2 \\
	&
	 - \frac{|A|^2}{|\bH|^2} \, \left(  |\nabla \bH |^2 - |\nabla |\bH| |^2 + \frac{1}{2} \, |\bH|^2
	- \langle \Delta \phi , \bH \rangle -  A_{im}^{\phi}  A_{im}^{\bH} - \left| A^{\bH} \right|^2 \right)  \notag  	 \, .
\end{align}
The lemma follows from this and using Corollary \ref{c:nablabN} to rewrite $ |\nabla \bH|^2 - |\nabla |\bH||^2$.
\end{proof}

\begin{proof}[Proof of Theorem \ref{c:kappa}]
   We have
\begin{align}	\label{e:kappadv}
	\e^f \, |\bH|^{-2} \,  \dv \, \left\{ \psi^2 \, |\bH|^2 \e^{-f} \, \nabla |\tau|^2 \,   \right\} &= \psi^2 \,  \cL_{|\bH|^2} \, |\tau|^2 + 2\psi \langle \nabla \psi , \nabla |\tau|^2 \rangle \notag \\
	&=  2\, \frac{\psi^2 }{|\bH|} \,  \langle A , \cL_{|\bH|^2} \,\tau \rangle + 2\, \psi^2 \, |\nabla \tau|^2 + 4 \psi \langle \nabla_{\nabla \psi } \tau , \tau  \rangle \\ 
	&\geq 2\, \frac{\psi^2 }{|\bH|} \,  \langle A , \cL_{|\bH|^2} \,\tau \rangle +   \psi^2 \, |\nabla \tau|^2 -  4\, |\nabla \psi|^2 \,   |\tau|^2   
	\notag \,  ,
\end{align}
where the last inequality used the absorbing inequality $4ab \leq a^2 + 4b^2$. 
Lemma \ref{l:kappa}
gives  
\begin{align}	\label{e:csiA}
	|\bH| \,\langle A ,  \cL_{|\bH|^2} \, \tau \rangle 
	 &=   P +     A_{jm}^{\phi}  \langle A_{im}, A_{ij} \rangle  +  \langle \Hess_{\phi} , A \rangle + \frac{|A|^2}{|\bH|} \left( \langle \Delta \phi , \bN \rangle
	 + A^{\phi}_{ij} A^{\bN}_{ij} \right) \notag \\
	 &+ \frac{|A|^2}{|\bH|^2} \, \left\{ \Tr \, A^{\bN} (x^T , \cdot ) \langle \nabla_{\cdot}^{\perp} \phi , \bN \rangle 
	+ \left| \langle \nabla^{\perp} \phi , \bN \rangle \right|^2 -   \Tr \, A^{ \nabla_{\cdot}^{\perp}\phi }(x^T , \cdot )    - \left|  \nabla^{\perp} \phi \right|^2   
	 \right\} \\
	 &\geq P -  C \, \left(|A| + \frac{|A|^2}{|\bH|} \right) \,\left( |A|^2 \, |\phi| +  \left| \Hess_{\phi} \right| \right) - C \, \frac{|A|^2\, |\nabla^{\perp} \phi|}{|\bH|^2} \left( |A(x^T , \cdot)| + |\nabla^{\perp} \phi| \right)
	  \, .  \notag
\end{align}
Finally, the divergence theorem, \eqr{e:kappadv} and \eqr{e:csiA} give \eqr{e:ckappa}.
\end{proof}

 \section{Jacobi fields on a cylinder}		\label{s:sectionJ}
 
Throughout this section $\Sigma =  \SS^k_{\sqrt{2k}} \times \RR^{n-k} \subset \RR^N$, for some $k \geq 1$, is a fixed cylinder.   
Let $\theta$ be coordinates on the sphere $\SS^k$, $y_i$ be coordinates on the axis
$ \RR^{n-k}$, and $z_{\alpha}$ be coordinates on the remaining Euclidean directions $\RR^{N-n-1}$.  
 We will use that the first non-zero eigenvalue of the Laplacian on $\SS^k_{\sqrt{2k}}$ is $\frac{1}{2}$, the $\frac{1}{2}$-eigenspace is spanned by the coordinate functions $x_i$ for $i= 1 , \dots , k+1$, and
 the next eigenvalue is strictly greater than one.

\vskip1mm
The next proposition identifies the Jacobi fields on $\Sigma$. In the proposition and the corollary that follows, $V$ is a normal vector field on $\Sigma$.

\begin{Pro}	\label{l:jacobi}
If     $L \, V = 0$ and    $\|  V \|_{W^{2,2}} < \infty$, then
\begin{align}	
	V = \left\{ \sum_j   y_j \, f_j (\theta) + \sum_{i \leq j} b_{ij} (y_i y_j - 2 \, \delta_{ij}) \right\} \, \bN + \sum_{\alpha} \left\{  f^{\alpha}(\theta) + \sum_j  a_j^{\alpha} \, y_j
	\right\} \, \partial_{z_{\alpha}}
	\, ,
\end{align}
where each $f_j$ and $f^{\alpha}$ is a $\frac{1}{2}$-eigenfunction on $\SS^k$ and $ a^{\alpha}_j , b_{ij} \in \RR$.
\end{Pro}

\begin{proof}
The operator $L$ on $\Sigma$ becomes
\begin{align}	\label{e:2p3}
	L = \cL + \frac{1}{2} + \frac{1}{2} \, \Pi_{\bN} \, , 
\end{align}
where $\Pi_{\bN}$ is orthogonal projection onto the principal normal $\bN$.

The normal space to $\Sigma$ is spanned by   $N-n$ (pointwise) orthonormal parallel   vector fields
\begin{align}
	\{ \bN , \partial_{z_{\alpha}} \,  | \, \alpha = 1 , \dots , N-n-1 \} \, .
\end{align}
If we decompose $V$ into $V= V^0 \, \bN + \sum_{\alpha} V^{\alpha} \, \partial_{z_{\alpha}}$, then this and \eqr{e:2p3} give
\begin{align}	\label{e:decoV}
	L \, V = \left\{ (\cL + 1) V^0   \right\} \, \bN + \sum_{\alpha}  \left\{ (\cL + \frac{1}{2}) V^{\alpha}   \right\} \,  \partial_{z_{\alpha}} \, .
\end{align}
In particular, we see that $V^0$ is an eigenfunction with eigenvalue $1$ for $\cL$, while the $V^{\alpha}$'s are eigenfunctions with eigenvalue $\frac{1}{2}$.  It follows from lemma $3.26$ in \cite{CM3} (and its proof{\footnote{Lemma $3.26$ in \cite{CM3} deals with eigenvalue $1$; obvious modifications give eigenvalue $\frac{1}{2}$ as well.}}) that:
\begin{itemize}
\item If $u \in W^{2,2}$ has $\cL u = - \frac{1}{2} \, u$, then $u = f(\theta) + \sum a_j \, y_j$ where $a_j \in \RR$ and $f$ is an eigenfunction on $\SS^k$ with eigenvalue $\frac{1}{2}$.
\item  If $u \in W^{2,2}$ has $\cL u = -  u$, then $u =  \sum   y_j \, f_j (\theta) + \sum_{i \leq j} b_{ij} (y_i y_j - 2 \, \delta_{ij})$ where $b_{ij} \in \RR$ and each $f_j$ is an eigenfunction on $\SS^k$ with eigenvalue $\frac{1}{2}$.
\end{itemize}
The proposition follows.
\end{proof}

\begin{Cor}	\label{c:jacobi}
If     $L \, V = 0$ and    $\|  V \|_{W^{2,2}} < \infty$, then there exist $b_{ij} \in \RR$ and a rotation   $\bar{V}$ so that
\begin{align}	
	V = \bar{V}^{\perp} + \left\{  \sum_{i \leq j} b_{ij} (y_i y_j - 2 \, \delta_{ij}) \right\} \, \bN 	\, .
\end{align}
\end{Cor}

\begin{proof}
The rotation in the $x_i-x_j$ plane is given by $x_i \partial_{x_j} - x_j \partial_{x_i}$.  The normal part of this is
\begin{align}
	\bar{V}^{\perp}_{x_i , x_j} = x_i \partial_{x_j}^{\perp} - x_j \partial_{x_i}^{\perp} \, .
\end{align}
There are six cases to consider, depending on whether $x_i , x_j$ are spherical coordinates, axis coordinates, or in the $z_{\alpha}$'s.  Three of these are zero (when both are spherical, both are along the axis, 
or both are in the $z_{\alpha}$ directions).  The remaining three cases are:
\begin{itemize}
\item If $x_i$ is in $\RR^{k+1}$ and $x_j = y_j$, then $\bar{V}^{\perp}_{x_i , y_j} = - y_j f_i (\theta) \, \bN$.
\item If $x_i$ is in $\RR^{k+1}$ and $x_j = z_j$, then $\bar{V}^{\perp}_{x_i , z_j} =  f_i (\theta) \, \partial_{z_j}$.
\item If $x_i = y_i$ and $x_j = z_j$, then $\bar{V}^{\perp}_{y_i , z_j} =  y_i \, \partial_{z_j}$.
\end{itemize}
Linear combinations of these span the first, third and fourth terms in the expression for $V$ in Proposition \ref{l:jacobi}, giving the corollary.
\end{proof}

\subsection{An effective decomposition}

Let $\Lambda_{\frac{1}{2}}$ be the linear space of $\frac{1}{2}$ eigenfuctions on $\SS^k_{\sqrt{2k}}$.

\begin{Lem}	\label{l:eigenA}
There exists $C$ so that 
if $u$ is a function on $\Sigma$ with $\| u \|_{W^{2,2}} < \infty$, then:
\begin{itemize}
\item  There exists $f(\theta) \in \Lambda_{\frac{1}{2}}$ and constants $a_j$ with 
\begin{align}
	\left\| u - f - \sum_j a_j y_j  \right\|_{W^{2,2}} \leq C \, \| (\cL + \frac{1}{2}) u \|_{L^2} \, .
\end{align}
\item There exist $f_j (\theta) \in  \Lambda_{\frac{1}{2}}$ and constants $b_{ij}$ with 
\begin{align}
	\left\| u - \sum_j f_j y_j - \sum_{i\leq j} b_{ij} (y_i y_j - \delta_{ij}) \right\|_{W^{2,2}} \leq C \, \| (\cL +1) u \|_{L^2} \, .
\end{align}
\end{itemize}
\end{Lem}

\begin{proof}
This is  standard; see, for instance, lemma  $3.2$   in \cite{CM3}.
\end{proof}

\begin{Cor}	\label{c:eff}
There exists $C$ so that if $V$ is normal with $\| V\|_{W^{2,2}} < \infty$, then the $L^2$-orthogonal projection $J$ of $V$ onto the  Jacobi fields described in  Proposition \ref{l:jacobi}  satisfies
\begin{align}
	\| V - J \|_{W^{2,2}} \leq C \, \| L \, V \|_{L^2} \, .
\end{align}
\end{Cor}

\begin{proof}
This follows immediately from \eqr{e:decoV} and Lemma \ref{l:eigenA}.

\end{proof}

We will later need that on the space of Jacobi fields, the pointwise norms can be bounded in terms of the $L^2$ bound on a fixed ball as follows:

\begin{Lem}	\label{l:Jbound}
There exists $C$ depending on $N$ so that if $J$ is a Jacobi field as in Proposition \ref{l:jacobi}, then at $x$
\begin{align}
	|J|   &\leq C \, (1+ |x|^2) \, \| J \|_{L^2 (B_{\sqrt{2n}+1})} \,  , \\
	  |\nabla  J|  +   |\Hess_J|  &\leq C \, (1+ |x|) \, \| J \|_{L^2 (B_{\sqrt{2n}+1})} \, , \\
	  \left| \Hess_J ( \cdot , \RR^{n-k}) \right| &\leq C \,  \| J \|_{L^2 (B_{\sqrt{2n}+1})} \, .
\end{align}
\end{Lem}

\begin{proof}
To see this, note that there is a finite-dimensional basis of Jacobi fields, each of which grows at most quadratically.  Moreover, the restriction of these to $B_{\sqrt{2n}+1}$ is injective and, thus, $L^2 (B_{\sqrt{2n}+1})$ is a norm on this space.  For the last claim, note that the Jacobi fields are polynomials in the $\RR^{n-k}$ variables, either of degree at most one (possibly times a function of $\theta$) or of degree two with no $\theta$ dependence.
\end{proof}

\section{Classification}	\label{s:class}

The tensor $\tau$ is parallel on a cylinder.  The main result of this section, Proposition \ref{p:Lojae1A},    bounds the distance to a cylinder in terms of $\nabla \tau$ and $\phi$, 
{\emph{as long as it is close to a cylinder on some fixed large ball}}.  This is a quantitative almost rigidity for the cylinder.   
In codimension one, a shrinker with $\nabla \tau = 0$ must be a sphere, cylinder or $\gamma \times \RR^{n-1}$ for an Abresch-Langer curve $\gamma$.  A quantitative form of this classification was used  in \cite{CM3} to bound the distance to cylinders for hypersurfaces, but  that argument relied   on having only one normal direction.  In fact,  there are non-cylindrical shrinkers with $\nabla \tau = 0$ in higher codimension.

\vskip1mm
 A key  will be to understand the structure of the symmetric
   real-valued tensor $\tau^{\bN} = \langle \tau , \bN \rangle$.  The eigenvalues{\footnote{By convention, an operator $T$ has eigenvalue $\lambda$ if there is a nonzero vector $x$ with $Tx+ \lambda x =0$.}} of $\tau^{\bN}$ on $\SS^k_{\sqrt{2k}} \times \RR^{n-k}$ are
 $0$ and $\frac{1}{k}$ with multiplicities $(n-k)$ and $k$, respectively.
The norm $|\bH|$   is constant on a cylinder.  The next lemma shows that   $\nabla |\bH|$ and $\Hess_{|\bH|}$ are almost eigenvectors of $\tau^{\bN}$ with eigenvalue one.  Thus, if $|\bH|$ is not constant (and $k>1$), then we  get eigenvalues of $\tau^{\bN}$  that are   far from those of the cylinder.

\begin{Lem}	\label{l:A1}
There exists $C$ depending on $n$ so that if  $\bH \ne 0$, then 
\begin{align}
	\left| \tau^{\bN} (\nabla |\bH|) + \nabla |\bH|  \right| &\leq  n \, |\bH| \, |\nabla \tau |  \, , \\
	\left| \tau^{\bN} \circ \Hess_{ |\bH|} +  \Hess_{ |\bH|}  \right| &\leq C \,   |\nabla |\bH|| \, |\nabla \tau| \, (1+ |\tau|) + C \, |\bH| \, ( |\nabla \tau| + |\Hess_{\tau}| ) \, .
\end{align}
\end{Lem}

\begin{proof}
Let $E_i$ be an orthonormal frame for $\Gamma$.
Since $A= |\bH| \, \tau$, we have 
\begin{align}
\nabla_{E_i} A_{jk}  = \left( \nabla_{E_i}   |\bH| \right) \, \tau_{jk} + |\bH| \, \nabla_{E_i} \tau_{jk} \, .
\end{align}
By Codazzi, $\nabla A$ is fully symmetric in $i, j , k$.  Taking the trace over $i=j$, we see that
\begin{align}
-\nabla_{E_k} \bH = \nabla_{E_k} A_{ii} = \nabla_{E_i} A_{ik} = \left( \nabla_{E_i}   |\bH| \right) \, \tau_{ik}  + |\bH| \, \nabla_{E_i} \tau_{ik} 
	= \left( \nabla_{E_i}   |\bH| \right) \, \tau_{ik}  + |\bH| \, ( \dv \, \tau ) (E_k) \, . \notag
\end{align}
Taking the inner product with $\bN$ and using that $\langle \nabla \bN , \bH \rangle = 0$ gives 
\begin{align}	\label{e:diffthis}
	  |\bH|_i \, \tau_{ik}^{\bN} &= - \langle \nabla_{E_k} \bH , \bN \rangle - |\bH| \, \langle  \dv ( \tau ) (E_k)  , \bN \rangle = - |\bH|_k   - |\bH| \, \langle  \dv ( \tau ) (E_k)  , \bN \rangle
	 \, .
\end{align}
Since $|\dv \, \tau| \leq n \, |\nabla \tau|$, this gives the first claim.
   For the second claim, we choose the frame so that $\nabla_{E_i}^T E_j = 0$ at the point and we differentiate \eqr{e:diffthis} to get
\begin{align}	\label{e:diffthis2}
	  |\bH|_{ij} \, \tau_{ik}^{\bN}  & +  |\bH|_i \, \langle  \nabla_{E_j} \tau_{ik}  , \bN \rangle +  |\bH|_{i} \, \tau_{ik}^{\bN_j^{\perp}}= - |\bH|_{kj}  - |\bH|_j \, \langle  \dv ( \tau ) (E_k)  , \bN \rangle   
	   \notag \\
	 & - |\bH| \, \langle  \nabla_{E_j} \dv ( \tau ) (E_k)  , \bN \rangle    - |\bH| \, \langle  \dv ( \tau ) (E_k)  , \bN_j \rangle
	 \, .
\end{align}
In particular, using also that $\bN = - \Tr \, \tau$ to bound derivative of $\bN$ by derivatives of $\tau$,  there is a constant $C$ depending on $n$ so that
\begin{align}
		\left| |\bH|_{ij} \, \tau_{ik}^{\bN}  +  |\bH|_{kj} \right| \leq C \,   |\nabla |\bH|| \, |\nabla \tau| \, (1+ |\tau|) + C \, |\bH| \, ( |\nabla \tau| + |\Hess_{\tau}| ) \, .
\end{align}
\end{proof}

 The next lemma shows that if $\nabla |\bH|$,  $\nabla \tau$  and $\phi$ are small, then $A_{ij} \approx - 2  \, |\bH| \, A_{ik}   A^{\bN}_{kj}$.  In particular, $A$ is almost an eigenvector of $A^{\bN}$ with eigenvalue $2 |\bH|$.

\begin{Lem}	\label{l:A1a}
There exists $C$ depending on $n$ so that if   $\bH \ne 0$,     then 
 \begin{align}	 	\label{e:myA1a}
	    \left| A_{ij}     
	+2\,   |\bH| \, A_{ik} \,  A^{\bN}_{jk} -  2\frac{|\bH|_i \,|\bH|_j}{|\bH|} \,  \bN
		+2\,  \bN \,  \Hess_{|\bH|} (E_i , E_j)  \right|   \leq C\, |\bH| \left( |\nabla \tau | \, |x^T| + |\nabla^2 \tau| \right) \notag \\
		+ 2\, \frac{|\nabla \phi| \, |\nabla |\bH||}{|\bH|}+ 2 \, |\Hess_{\phi}| + C \,
		|\nabla |\bH|| \, |\nabla \tau| + C \, |A| \, |A^{\phi}|
		\, .
\end{align}
\end{Lem}

\begin{proof}
Using \eqr{e:231}, the Codazzi equation $ \left( \nabla_{x^T} A\right)_{ij} =  \left( \nabla_{E_i} A\right) (x^T,E_j)$ and   $A = |\bH| \, \tau$ gives
\begin{align}	 \label{e:411e}
	  - &\Hess_{\bf{H}}   (E_i, E_j)  =    \Hess_{\phi} (E_i, E_j) +  \frac{1}{2} \,  \left( \nabla_{x^T} A\right)_{ij}
	+ \frac{1}{2}  A_{ij} +  \frac{1}{2} \, A_{ik} \,  A^{x^{\perp}}_{jk}\notag  \\
		&=   \Hess_{\phi}  (E_i, E_j) + \frac{ |\bH|_i}{2} \,  \tau (x^T , E_j) + \frac{ |\bH| }{2}\, \nabla_{E_i} \tau (E_j , x^T)
	+ \frac{1}{2}  A_{ij} +  A_{ik} \,  A^{\bH}_{jk}+ A_{ik} \,  A^{\phi}_{jk}
 \, .  
\end{align}
Lemma \ref{l:gradbH}
gives that  $\nabla_{E_j}^{\perp} \bH = 
	- \frac{1}{2} A(x^T , E_j ) - \nabla_{E_j}^{\perp} \phi $ and, thus, 
 \begin{align} 	\label{e:e411}
	 \frac{1}{2} \,  |\bH|_i \,  \tau (x^T , E_j) =  - \frac{ |\bH|_i}{|\bH|} \, \left(  \bH_j^{\perp}  +  \phi_j^{\perp} \right) =  - \frac{ |\bH|_i}{|\bH|} \, \left( |\bH|_j \, \bN + |\bH| \, \bN_j^{\perp}  +  \phi_j^{\perp} \right) \, .
\end{align}
On the other hand, differentiating $\bH = |\bH| \, \bN$ twice gives that
\begin{align}	\label{e:e412}
	\Hess_{\bf{H}}  (E_i, E_j) =\bN \,  \Hess_{|\bH|} (E_i , E_j) + |\bH|_i \nabla_{E_j}^{\perp} \bN +  |\bH|_j \nabla_{E_i}^{\perp} \bN  + |\bH| 
	\, \Hess_{\bN} (E_i , E_j) \, .
\end{align}
Using \eqr{e:e411} and the formula \eqr{e:e412} for $\Hess_{\bH}$ in \eqr{e:411e} gives
 \begin{align}	 	\label{e:e413}
	   A_{ij}  &=   -    |\bH| \,   \nabla_{E_i}  \tau (x^T , E_j)
	-2\,   |\bH| \, A_{ik} \,  A^{\bN}_{jk}  + 2  \frac{ |\bH|_i}{|\bH|} \, \left( |\bH|_j \, \bN + |\bH| \, \bN_j^{\perp}  +  \phi_j^{\perp} \right)  \notag \\
	&- 2\, A_{ik} A^{\phi}_{jk} - 2 \, \Hess_{\phi} (E_i , E_j)  -2\,  \bN \,  \Hess_{|\bH|} (E_i , E_j) - 2\, |\bH|_i \nabla_{E_j}^{\perp} \bN   \\
	&-2\,  |\bH|_j \nabla_{E_i}^{\perp} \bN  -2\, |\bH| 
	\, \Hess_{\bN} (E_i , E_j)
 \, . \notag 
\end{align}
The inequality \eqr{e:myA1a}  follows from \eqr{e:e413} and the observation that
 $\bN =-\Tr \, \tau$  so $|\nabla \bN|$ and $|\Hess_{ \bN}|$ are bounded by $|\nabla \tau|$ and $|\nabla^2 \tau|$.
\end{proof}

\subsection{The approximate kernel of $\tau^{\bN}$}

The next lemma estimates the eigenvalues of the  two-tensor $\tau^{\bN}$.  The ``error terms'' on the right vanish if $\Gamma$ is a shrinker and $\tau$ is parallel.

\begin{Lem}	\label{l:appkern}
There exists $C$ depending on $n$ so that if $\bH (p) \ne 0$ and   $\tau^{\bN}(E_1 , \cdot) = \lambda \, E_1$ at $p \in \Gamma$ with $|\lambda| \leq \frac{3}{4}$, then
 \begin{align}	 \label{e:appker1}
	  |\lambda| \,   \left| 1
	+2\,   |\bH|^2 \,  \lambda  \right|  &\leq   C\,   |\nabla \tau | \, (1+|x^T|) + C \, |\nabla \tau |^2+ C\,  |\nabla^2 \tau|   + C\, |\tau| \, |A^{\phi}| \notag \\
		&+ 2\, \frac{|\nabla \phi| \, |\nabla \log |\bH|| +   |\Hess_{\phi}|}{|\bH|} + C \,
		\frac{|\nabla |\bH||}{|\bH|} \, |\nabla \tau|  (1+ |\tau|)
		\, .
\end{align}
\end{Lem}

\begin{proof}
Choose   $E_i$ for $i \geq 2$ to diagonalize $\tau^{\bN}$ at $p$, so that 
\begin{align}
	A^{\bN}_{1j} = |\bH| \, \tau^{\bN}_{1j} = |\bH| \, \lambda \, \delta_{1j} \, .
\end{align}
Taking $i=j=1$ in Lemma \ref{l:A1a} and then taking the inner product with $\bN$
gives  
 \begin{align}	 \label{e:my418}
	    \left| |\bH| \, \lambda  
	+2\,   |\bH|^3 \,  \lambda^2 -  2\frac{|\bH|_1^2}{|\bH|}  
		+2\,    \Hess_{|\bH|} (E_1 , E_1)  \right|   \leq C\, |\bH| \left( |\nabla \tau | \, |x^T| + |\nabla^2 \tau| \right) \notag \\
		+ 2\, \frac{|\nabla \phi| \, |\nabla |\bH||}{|\bH|}+ 2 \, |\Hess_{\phi}| + C \,
		|\nabla |\bH|| \, |\nabla \tau| + C \, |A| \, |A^{\phi}|
		\, .
\end{align}
Lemma \ref{l:A1}
gives that
\begin{align}
	\left| (\lambda +1)\,  |\bH|_1  \right| &\leq 2n \, |\bH| \, |\nabla \tau |  \, , \\
	\left| (\lambda + 1) \Hess_{ |\bH|}(E_1 , E_1)   \right| &\leq C \,   |\nabla |\bH|| \, |\nabla \tau| \, (1+ |\tau|) + C \, |\bH| \, ( |\nabla \tau| + |\Hess_{\tau}| ) \, .
\end{align}
Since  $|\lambda| < \frac{3}{4}$, we conclude that
\begin{align}
	\left|   |\bH|_1  \right| &\leq 8n \, |\bH| \, |\nabla \tau |  \, , \\
	\left|  \Hess_{ |\bH|}(E_1 , E_1)   \right| &\leq C \,   |\nabla |\bH|| \, |\nabla \tau| \, (1+ |\tau|) + C \, |\bH| \, ( |\nabla \tau| + |\Hess_{\tau}| ) \, .
\end{align}
Using these bounds in \eqr{e:my418} and dividing by $|\bH|$ gives \eqr{e:appker1}.
\end{proof}

In applications,  $\Gamma$ will be  close to a cylinder and   $\nabla \tau$ and $\phi$  very small.  We will   show that $\Gamma$ is then very close to a cylinder.  The eigenvalues of $\tau^{\bN}$ on the cylinder are $0$ and $\frac{1}{k}$.  The next corollary shows that if $\Gamma$ has an eigenvalue below $\frac{1}{k}$, then it must be very close to zero.

\begin{Cor}	\label{c:appkern}
There exists $C$ depending on $n$ so that if $|\bH|^2 \geq \frac{1}{4} $ and $ 1 \geq |\nabla |\bH||$  at  $p \in \Gamma$ and $\lambda$ is an eigenvalue of $\tau^{\bN}$ at $p$ with  $|\lambda| \leq \min \{ \frac{3}{4} , \frac{1}{4|\bH|^2} \}$, then
 \begin{align}	 
	  |\lambda|   &\leq   C\,  \left( |\nabla \tau | \, \left[   |x^T| +       |\tau| + |\nabla \tau| \right] + |\nabla^2 \tau| + |A|^2 \, |\phi| +  |\nabla \phi|   +   |\Hess_{\phi}| \right)  \,
		\, .  
\end{align}
\end{Cor}

\begin{proof}
Since $2|\bH|^2 \, |\lambda | \leq \frac{1}{2}$,   we  have that $\frac{1}{2} \leq \left| 1
	\pm 2\,   |\bH|^2 \,  \lambda  \right| $.  Thus, Lemma \ref{l:appkern} gives 
 \begin{align}	 
	  |\lambda|   &\leq   C\,  \left( |\nabla \tau |   \left[ 1+ |x^T| + \frac{|\nabla |\bH||}{|\bH|}     (1+ |\tau|) \right] + |\nabla \tau|^2 + |\nabla^2 \tau| \right) \notag \\
	  &+ 4\, \frac{|\nabla \phi| \, |\nabla \log |\bH||   +   |\Hess_{\phi}|}{|\bH|}  
	  + C \, |\tau | \, |A^{\phi}| \,
		\, .  
\end{align}
The corollary follows from this and the assumptions on $\bH$.
\end{proof}

The next lemma shows that if $\nabla \tau$ and $\phi$ are almost zero, then any vector almost in the kernel of $\tau^{\bN}$  is almost in the kernel of $A$.  This will use Lemma \ref{l:A1a}.

\begin{Lem}	\label{l:fullAsmall}
There exists $C$ depending on $n$ so that if $|\bH|^2 \geq \frac{1}{4} $,  $ 1 \geq |\nabla |\bH||$, and $4 \geq |A|^2$  at  $p \in \Gamma$, then
for any tangent vector $V$ at $p$
\begin{align}	 
	    \left| A (    V , \cdot ) \right| \leq 
	C\, \left| \tau^{\bN} (V) \right|  + C\,  \left( |\nabla \tau | \,(1+ |x^T|) + |\nabla^2 \tau|   + |\phi| +   |\nabla \phi|  +   |\Hess_{\phi}|   \right) \, |V|
		\, .
\end{align}
\end{Lem}

\begin{proof}
Given a normal vector $W$, let $W^{N^{\perp}}= W - \langle W , \bN \rangle \, \bN$ denote its projection orthogonal to $\bN$.
Let $E_i$ be an orthonormal frame.
Taking  the projection to $\bN^{\perp}$  in Lemma \ref{l:A1a} gives
 \begin{align}	 
	    \left| A_{ij}^{\bN^{\perp}}     
	+2\,   |\bH|^2 \, A_{ik}^{\bN^{\perp}}   \,  \tau^{\bN}_{jk}  
		   \right|   \leq C\,  \left( |\nabla \tau | \,(1+ |x^T|) + |\nabla^2 \tau|   + |\phi| +  |\nabla \phi|  +   |\Hess_{\phi}|   \right)
		\, .
\end{align}
Applying this to $V= V_i \, E_i$ gives
\begin{align}	 
	    \left| A_{ij}^{\bN^{\perp}}     V_i \right| \leq 
	C\, \left| \tau^{\bN} (V) \right|  + C\,  \left( |\nabla \tau | \,(1+ |x^T|) + |\nabla^2 \tau|   +|\phi|  +   |\nabla \phi|  +   |\Hess_{\phi}|   \right) \, |V|
		\, .
\end{align}
The remaining part of $A_{ij}      V_i$ is in the $\bN$ direction, which is controlled by $\tau^{\bN}(V)$.
\end{proof}

\subsection{A slicing lemma}

To understand the next lemma, it is useful to review some of the properties of a product shrinker $  \Sigma_0^k \times \RR^{n-k} \subset \RR^N$:
\begin{itemize}
\item   There is an $(n-k)$-dimensional Euclidean subspace $\cV$ of translation invariance; in particular, the normal projection of any $v \in \cV$ is zero.  
\item If we   intersect  $\Sigma$ with the (Euclidean) orthogonal complement of $\cV$, then we get the $k$-dimensional shrinker $\Sigma_0$ in $\RR^{N-(n-k)}$. 
\item 
The coordinate functions on $\Sigma_0$ are $\frac{1}{2}$-eigenfunctions of the new drift Laplacian $\cL_0$ and $|x|^2-2k$ is a $1$-eigenfunction.  
\end{itemize}
The next lemma shows that we almost get these properties  for a submanifold $\Gamma^n \subset \RR^N$  that is almost invariant under translation by a subspace.  In the lemma, $\Pi$ is orthogonal projection onto the normal space to $\Gamma$, and
$z_1, \dots , z_N$ is  an orthonormal set of coordinate functions on $\RR^N$.  Let $\cV$ be the span of $\partial_{k+2} , \dots , \partial_{n+1}$.

\begin{Lem}	\label{l:430}
If  $\delta > 0$ is small,  $\Gamma_0 = \Gamma \cap \{ z_{k+2} =  \dots = z_{n+1} = 0 \} $ is transverse so that 
	$ |\Pi (v)| \leq \delta \, |v| {\text{ on }} \Gamma_0$  for all $v\in \cV$ 
	 and $\cL_0$ is the drift Laplacian on $\Gamma_0$, then  
\begin{align}
	\left| (\cL_0 +1) \, (2k- |z|^2 ) \right| &\leq 2 \,   |z|  \, |\phi|   +4\, k \,\delta \,  |z| \, |A|   + 2(n-k) \, \left( \delta^2 \, |z|^2 +   |z| \, \sup \left| A |_{\Gamma_0^{\perp}} \right|  \right) \, ,  \label{e:wisz2}  \\
	\left| \left( \cL_0 + \frac{1}{2} \right) \, z_j \right| &\leq   |\phi| + \delta \, |z| + (n-k) \, \sup  \left| A |_{\Gamma_0^{\perp}} \right| + 2\, k \, \delta \, |A|
	\, .  \label{e:wisz2a}
\end{align}
\end{Lem}

\begin{proof}
Choose an orthonormal frame $E_1 , \dots , E_n$ for $\Gamma$ with $E_1 , \dots , E_k$ tangent to $\Gamma_0$.    Let   $\Pi_0$ be orthogonal projection onto the normal space to $\Gamma_0$.   
For each $p \in \Gamma_0$, let $W_p$ be the orthogonal complement of $T_p\Gamma_0$ inside of $T_p \Gamma$.  In particular, $W_p$ is spanned by
$E_{k+1} , \dots , E_n$ and $W_p$ has dimension $n-k$.  We will omit the $p$ and simply write $W$ when it is clear.

The map $v \in \cV$ goes to $v^T$ is injective  since $ |\Pi (v)| \leq \delta \, |v| {\text{ on }} \Gamma_0$.  Given $i \leq k$, we have
\begin{align}
	\langle E_i , v^T \rangle = \langle E_i , v \rangle - \langle E_i , \Pi (v) \rangle = 0 \, .
\end{align}
Thus, if $v \in \cV$, then $v^T \in W$.  Since the dimensions are the same, this map is bijective.

Given a function $w$, we have
\begin{align}	 \label{e:foranyw}
	\Delta_0 \, w &= \sum_{i=1}^k \langle \nabla_{E_i} \left( \nabla^T w - \Pi_0 (\nabla^T w) \right) , E_i \rangle =  \sum_{i=1}^k \langle \nabla_{E_i}   \nabla^T w   , E_i \rangle
	-  \sum_{i=1}^k \langle \nabla_{E_i}   \Pi_0 (\nabla^T w)   , E_i \rangle \notag \\
	&=  \Delta w - \sum_{i=k+1}^n \langle \nabla_{E_i}   \nabla  w   , E_i \rangle +  \sum_{i=k+1}^n \langle \nabla_{E_i}   \Pi (\nabla  w  ) , E_i \rangle
	+  \sum_{i=1}^k \langle   \Pi_0 (\nabla^T w)   , \nabla_{E_i}  E_i \rangle
	\, . % \notag
\end{align}
Since $\cL_0 w = \Delta_0 w - \frac{1}{2} \langle z  - \Pi_0(z) , \nabla w \rangle$ and $\cL w = \Delta w - \frac{1}{2} \langle z  - \Pi (z) , \nabla w \rangle$, we have
\begin{align}	\label{e:combinez}
	\cL_0 \, w  =   \cL \, w +  \Delta_0 w - \Delta w + \frac{1}{2} \langle    \Pi_0(z) - \Pi (z)  , \nabla w \rangle \, .
\end{align}
Now set $w=|z|^2$.  Since the Euclidean Hessian of $|z|^2$ is twice the identity,   \eqr{e:foranyw} gives
\begin{align}	\label{e:delta0delta}
	\Delta_0 w &=  \Delta w - 2(n-k) -  \sum_{i=k+1}^n \langle   \Pi (\nabla  w  ) ,  \nabla_{E_i}  E_i \rangle
	+  \sum_{i=1}^k \langle   \Pi_0 (\nabla^T w)   , \nabla_{E_i}  E_i \rangle
	\, .  
\end{align}
Using that $\nabla \, |z|^2 = 2 \, z$ and
\begin{align}
	\cL \, |z|^2 &= 2\, \dv_{\Gamma} \, z^T - \langle z^T , z^T \rangle = 2n - 2\, \dv_{\Gamma} \, \Pi (z) -|z^T|^2 = 2n - 2 \langle \Pi (z) , \bH \rangle - w + |\Pi (z)|^2 \notag \\
	& = 2n - w + 2 \langle \Pi (z) , \phi \rangle \, ,
\end{align}
we combine \eqr{e:delta0delta} and \eqr{e:combinez} to get
\begin{align}	\label{e:summingup}
	\cL_0 \, |z|^2  &= 2k - |z|^2 + 2 \langle \Pi (z) , \phi \rangle +  \langle    (\Pi_0  - \Pi) (z)  , z  \rangle -  2\, \sum_{i=k+1}^n \langle   \Pi (z  ) ,  \nabla_{E_i}  E_i \rangle \notag \\
	&+  2\, \sum_{i=1}^k \langle   \Pi_0 (z^T )   , \nabla_{E_i}  E_i \rangle
\, .
\end{align}
We will  bound   the last four terms in \eqr{e:summingup} to prove \eqr{e:wisz2}.
First, we have
\begin{align}	\label{e:my3b1}
	\left|  \sum_{i=k+1}^n \langle   \Pi ( z  ) ,  \nabla_{E_i}  E_i \rangle  \right| \leq   (n-k) \, |z| \sup_{\eta \in W} \, \frac{|A(\eta,\eta)|}{|\eta |^2} \, .
\end{align}
Since $\langle    \Pi_0(z) - \Pi (z)  , z  \rangle = \sum_{i=k+1}^n \langle z , E_i \rangle^2$ and the map $v \to v^T$ from $\cV$ to $W$ is onto, we have
\begin{align}	\label{e:my3b2}
	\left| \langle    \Pi_0(z) - \Pi (z)  , z  \rangle \right| &\leq  (n-k) \,    \sup_{v\in \cV} \, \langle z , \frac{v^T}{|v^T|} \rangle^2 \leq 
	(n-k) \,  2 \, \sup_{v\in \cV} \, \langle z , \frac{v - \Pi (v)}{|v|} \rangle^2 \notag \\
	&= (n-k) \,  2 \, \sup_{v\in \cV} \, \langle z ,   \frac{\Pi (v)}{|v|} \rangle^2  \leq 2(n-k) \, \delta^2 \,  |z|^2 \, .
\end{align}
Here the equality used that $\langle v , z \rangle = 0$ for $v \in \cV$.

Observe next that $\Pi_0$ of any tangent vector must be tangent to $\Gamma$ but normal to $\Gamma_0$; in particular, $\Pi_0 (z^T  ) \in W$.  Thus, we can choose $v \in \cV$ so that $|v| \leq 2\, |z|$ and $v^T = \Pi_0 (z^T  )$.  Note that $\langle v^T , E_i \rangle = 0$ for $i \leq k$.  Thus, we get for each $i \leq k$ that
\begin{align}	\label{e:e442}
	\left| \langle   \Pi_0 (z^T  )   , \nabla_{E_i}  E_i \rangle \right| &= \left|  \langle   v^T   , \nabla_{E_i}  E_i \rangle \right|  
	 = \left|  \langle   \nabla_{E_i}   v^T   , E_i \rangle \right| \notag \\ &= \left|  \langle   \nabla_{E_i}   \Pi (v)   , E_i \rangle \right| 
	=  \left|  \langle    \Pi (v)   ,  \nabla_{E_i}  E_i \rangle \right| \leq 2\, |z| \, \delta \, |A| \, .
\end{align}
Using this, \eqr{e:my3b1} and \eqr{e:my3b2}
in \eqr{e:summingup} gives \eqr{e:wisz2}.

To get the second claim, first apply \eqr{e:foranyw} with $w=z_j$ to get
\begin{align}	 \label{e:e443}
	\Delta_0 \, z_j &=  \Delta \, z_j   -  \sum_{i=k+1}^n \langle   \Pi (\partial_{z_j}  ) , \nabla_{E_i}  E_i \rangle
	+  \sum_{i=1}^k \langle   \Pi_0 (\partial_{z_j}^T)   , \nabla_{E_i}  E_i \rangle
	\, .   
\end{align}
Applying the $\cL$ operator to $z_j$ on $\Gamma$ gives
\begin{align}	 \label{e:e444}
	\cL \, z_j = \Delta \, z_j - \frac{1}{2} \langle \partial_{z_j}^T , z  \rangle = - \langle \partial_{z_j}^{\perp} , \bH \rangle - \frac{1}{2} \langle \partial_{z_j}  , z  \rangle +   \frac{1}{2} \langle \partial_{z_j}^{\perp} , z  \rangle
	=  \langle \partial_{z_j} , \phi \rangle - \frac{1}{2} \, z_j \, .
\end{align}
Using \eqr{e:combinez} with $w= z_j$ and then   \eqr{e:e443} and \eqr{e:e444} gives
\begin{align}	
	\cL_0 \, z_j + \frac{1}{2} \, z_j &
 %=    \langle \partial_{z_j} , \phi \rangle  +  \Delta_0 z_j - \Delta z_j + \frac{1}{2} \langle    (\Pi_0 - \Pi) (z)  , \partial_{z_j} \rangle \notag \\ &
	=   \langle \partial_{z_j} , \phi \rangle + \frac{1}{2} \langle    (\Pi_0 - \Pi) (z)   , \partial_{z_j} \rangle
	 -  \sum_{i=k+1}^n \langle   \Pi (\partial_{z_j}  ) , \nabla_{E_i}  E_i \rangle  
	+  \sum_{i=1}^k \langle   \Pi_0 (\partial_{z_j}^T)   , \nabla_{E_i}  E_i \rangle 
	 \, .  \notag
\end{align}
We will bound the terms on the right.
As in \eqr{e:e442},  $  \Pi_0 (\partial_{z_j}^T)  \in W$ and we can choose $v \in \cV$ so that $|v| \leq 2 $ and $v^T =   \Pi_0 (\partial_{z_j}^T) $.  We get for each $i \leq k$ that
\begin{align}	 
\left|  \langle   \Pi_0 (\partial_{z_j}^T)   , \nabla_{E_i}  E_i \rangle \right|
	&= \left|  \langle  v^T   , \nabla_{E_i}  E_i \rangle \right| 
%=   \left|  \langle  \nabla_{E_i}    v^T   , E_i \rangle \right| =   \left|  \langle  \nabla_{E_i}    \Pi (v)   , E_i \rangle \right| 
	=  \left|  \langle    \Pi (v)   ,  \nabla_{E_i}  E_i \rangle \right|  \leq \delta \, |v| \, |A| \leq 2 \, \delta \, |A| \, .
\end{align}
Since the  projections $\Pi$ and $\Pi_0$ are symmetric and $\Pi - \Pi_0$ is projection to $W$,  we have
\begin{align}	\label{e:e2tolast}
	\left| \langle   ( \Pi_0  - \Pi )(z)  , \partial_{z_j} \rangle \right| &  \leq |z| \, \left|  (\Pi_0  - \Pi  )(\partial_{z_j}) \right| \leq |z| \, \sup_{\eta \in W} \frac{ \left| \langle \eta , \partial_{z_j} \rangle \right| }{|\eta|} = 
	 |z| \, \sup_{v\in \cV} \frac{ \left| \langle v^T , \partial_{z_j} \rangle \right| }{|v^T|}
	\, .
\end{align}
In proving \eqr{e:wisz2a}, 
we can assume that $\partial_{z_j}$ is not in $\cV$ since otherwise $z_j \equiv 0$ on $\Gamma_0$.  It follows that $\langle v , \partial_{z_j} \rangle = 0$ for any $v \in\cV$ and, thus,
\eqr{e:e2tolast} gives
\begin{align}	\label{e:anotherterm1}
	\left| \langle    (\Pi_0  - \Pi) (z)  , \partial_{z_j} \rangle \right| &  \leq  
	 |z| \, \sup_{v\in \cV} \frac{ \left|  \Pi (v )  \right| }{|v^T|} \leq \frac{\delta}{1-\delta}\, |z| \leq 2 \, \delta \, |z|  
	\, .
\end{align}
Combining these bounds gives \eqr{e:wisz2a}, completing the proof.
\end{proof}

\subsection{Distance to a cylinder}
 
 The proposition will   bound  the distance from $\Gamma$ to a cylinder, assuming that $\Gamma$ is close enough to a cylinder in the fixed ball $B_{2n}$.
   We will  assume that $B_R \cap \Gamma$ satisfies the following bounds:
 \begin{enumerate}
 \item[($A\star$)]  $|\bH|^2 \geq \frac{1}{4} $,  $ 1 \geq |\nabla |\bH||$,  $4\geq |A|^2$ and there exists $C_{\ell}$ for each $\ell$ so that $|\nabla^{\ell} A| \leq C_{\ell}$.
 \end{enumerate}
We will use the following condition:
 \begin{itemize}
\item[($\star_{\epsilon_1 , r}$)]  There is a compactly supported 
  $U_r$ on $\SS^k_{\sqrt{2k}} \times \RR^{n-k}$ so that after a rotation of $\RR^N$ $B_r \cap \Gamma$ is contained in the graph of $U$ with $\| U_r \|_{C^2}\leq \epsilon_1$ and $\| U_r \|^2_{L^2 } \leq \frac{1}{2} \, \e^{ - \frac{r^2}{4}}$.
\end{itemize}

   \begin{Def}	\label{d:epsilonclose}
$\Sigma$ and $\Gamma$ are $(\epsilon , R,   C^2)$-close  if $B_R \cap \Sigma$ is the normal graph of a vector field $U$ over (a subset of) $\Gamma$ and $\| U \|_{C^2} \leq \epsilon$.
\end{Def}

  \begin{Pro}	\label{p:Lojae1A}
  There exists $\epsilon_0 = \epsilon_0 (n)$ so that
 given  $\epsilon_1$ and  $\alpha > 0$, there exists  $\bar{R}$ so that if 
   $\Gamma$ is $(\epsilon_0 , 2n, C^2)$-close to $\SS^k_{\sqrt{2k}} \times \RR^{n-k}$,  ($A\star$) holds
 on $B_R \cap \Gamma$ for $R \geq \bar{R}$ and
 \begin{align}	\label{e:bdtauphi}
 	\left( |\phi|_2 + |\nabla \tau|_1^2 \right) \, \e^{ - \frac{|x|^2}{4} } \leq  \e^{ - \frac{R^2}{4} } {\text{ on }} B_R \cap \Gamma \, , 
 \end{align}
  then ($\star_{\epsilon_1 , (1-\alpha)R}$) holds.
\end{Pro}

 \begin{proof} 
In the proof, $C$ will be a constant that depends only on $n$; it will be allowed to change from line to line.
 Fix an arbitrary point $p \in B_{2n} \cap \Gamma$ and let $E_i(p)$ be an orthonormal frame for $T_p\Gamma$ that diagonalizes $\tau^{\bN}$ at $p$, ordered so that 
 \begin{itemize}
 \item $\left|\lambda_i + \frac{1}{k} \right| \leq \bar \epsilon_0$ for $i= 1 , \dots , k$.
 \item $|\lambda_i| \leq \bar \epsilon_0$ for $i=k+1 , \dots , n$,
 \end{itemize}
 where $\bar \epsilon_0$ is a function of $\epsilon_0$ that vanishes at $0$.  For $\epsilon_0 > 0$ small enough, 
 Corollary \ref{c:appkern} gives  
 \begin{align}
 	|\lambda_i| \leq C\,   \left( |\nabla \tau|_1 + |\phi|_2 \right) \leq C \, \e^{ - \frac{R^2}{8} }  {\text{ for }} i= k+1 , \dots , n \, .
 \end{align}
 Let $\cV$ be the span of $E_i (p)$ for $i= k+1 , \dots , n$ and define
  tangential vector fields $v_i =[E_i (p)]^T$, so that $v_i (p) = E_i (p)$.  
 
Let $\gamma \subset B_{\sqrt{|q|^2+4n}} \cap \Gamma$ be a curve from $\gamma(0)=p$ to $\gamma(s)=q$ with $|\gamma'| =1$.
 We will show that if   $i=k+1 , \dots , n   $, then
 at $q$
 \begin{align}
 	\left| \Pi (E_i (p)) \right| & = \left| E_i (p) - v_i   \right|  \leq C \, \e^{ \frac{|q|^2-R^2}{8}}\,    (1 +s^2) \, , \label{e:vi1}  \\
 	\left| A ( v_i   , \cdot ) \right| &\leq C \, \e^{ \frac{|q|^2-R^2}{8}}\,    (1 +s^2) \label{e:vi2} \,  .
 \end{align}
Let $w$  be a parallel vector field on $\gamma$ with $w(0)= E_i (p)$.  
 By \eqr{e:bdtauphi} and ($A\star$),   
 $|\nabla_{\gamma'} \tau^{\bN} (w , \cdot)| \leq C \, \e^{ \frac{|q|^2-R^2}{8}}$ and, thus, 
$
 	  |\tau^{\bN} (w, \cdot)| \leq C \, \e^{ \frac{|q|^2-R^2}{8}}\, (1 +s)$.   Lemma \ref{l:fullAsmall} then implies that
  \begin{align}	\label{e:tauNw}
 	  |A(w, \cdot)| \leq C \, \e^{ \frac{|q|^2-R^2}{8}}\, (1 +s) \, .
 \end{align}
Since $w$ is parallel along $\gamma$, we see that $|\nabla_{\gamma'}w|=  |\nabla_{\gamma'}^{\perp} w| = |A (w, \gamma')|$ and, thus,
\begin{align}	\label{e:437}
	|w  - w(0)| \leq \int |\nabla_{\gamma'} w| \leq  C \, \e^{ \frac{|q|^2-R^2}{8}}\,    (1 +s^2)  \, .
\end{align}
Since   $w$ is tangential, we   have at $q$ that
\begin{align}
	 |w  - w(0)| \geq \min_{\nu \in T_q\Gamma_0} \, \, |\nu - w(0)| =  \left| v_i (q) - E_i (p) \right| \, .
\end{align}
Therefore, \eqr{e:vi1} follows from \eqr{e:437}.  To get \eqr{e:vi2}, observe that $|w  - v_i| \leq |w-w(0)|$ along $\gamma$ and, thus, 
\eqr{e:tauNw},  \eqr{e:437} and ($A\star$) give that
\begin{align}
	|A(v_i , \cdot)| \leq |A (w , \cdot)| + C \, |w-v_i| \leq C \, \e^{ \frac{|q|^2-R^2}{8}}\,    (1 +s^2)  \, .
\end{align}
   
   Let $\cV^{\perp} = \{ x \in \RR^N \, | \, \langle x , v \rangle = 0 {\text{ for all }} v \in \cV \}$ be the Euclidean orthogonal complement of $\cV$, set $\Gamma_0 = \cV^{\perp} \cap \Gamma$, and let $\cL_0$ 
   be the drift Laplacian on $\Gamma_0$.   The $\epsilon_0$ closeness to the cylinder guarantees that $\Gamma_0 \subset B_{2n}$.
   Now that we have \eqr{e:vi1} and \eqr{e:vi2}, Lemma \ref{l:430} gives that
\begin{align}
	\left| (\cL_0 +1) \, (2k- |z|^2 ) \right| &\leq 4n  \, |\phi|   +64 n^2 \,\delta  + 8n^2 (n-k) \, \delta^2   + 4n^2 \, \sup \left| A |_{\Gamma_0^{\perp}} \right|  \leq C \, \e^{ -\frac{ R^2}{8}} \, ,   \label{e:from430a} \\
	\left| (\cL_0 +1/2) \, z_j \right| &\leq   |\phi| + 2n\, \delta   + n \, \sup  \left| A |_{\Gamma_0^{\perp}} \right| + 4\, n \, \delta  \leq C \, \e^{ -\frac{ R^2}{8}} 
	\, .   \label{e:from430b}
\end{align}
 The $\epsilon_0$ closeness to the cylinder gives that $\cL_0$ is close to the Laplacian on  $\SS^k_{\sqrt{2k}}$.  The Laplacian  on $\SS^k_{\sqrt{2k}}$ does not have an eigenvalue at $1$ and has multiplicity $k+1$ for the eigenvalue $\frac{1}{2}$.  In particular, for $\epsilon_0$ small enough, linear elliptic theory gives:
 \begin{itemize}
 \item $\cL_0 + 1$ is invertible and for any function $u$ on $\Gamma_0$
 \begin{align}
 	\| u \|_{C^2(\Gamma_0)} \leq C \, \| (\cL_0 +1 ) u \|_{L^2(\Gamma_0)} \, .	\label{e:g0eiga}
 \end{align}
\item There is a $k+1$ dimensional space $\Lambda$ of functions on $\Gamma_0$ so that if $u$  satisfies $\int_{\Gamma_0} u \, \mu\, \e^{-f} = 0$ for all $\mu \in \Lambda$, then
\begin{align}
	\| u \|_{C^2(\Gamma_0)} \leq C \, \| (\cL_0 +1/2 ) u \|_{L^2(\Gamma_0)} \, . \label{e:g0eigb}
\end{align}
\end{itemize}
First, combine \eqr{e:from430a} and \eqr{e:g0eiga} to conclude that
\begin{align}
	\| 2k-|z|^2 \|_{C^2(\Gamma_0)} \leq C \, \| (\cL_0 +1 ) (2k-|z|^2) \|_{L^2(\Gamma_0)} \leq  C \, \e^{ -\frac{ R^2}{8}} \, .
\end{align}
Choose a codimension $(k+1)$ subspace $\cV_0$ in $\cV^{\perp}$ so that $\int_{\Gamma_0} \langle v , x \rangle\,  \mu\, \e^{-f} = 0 $
for each  $v \in \cV_0$  and all $ \mu \in \Lambda$.
Combine \eqr{e:from430b} and \eqr{e:g0eigb} to conclude that for each $v \in \cV_0$
\begin{align}
	\| \langle x , v \rangle \|_{C^2(\Gamma_0)} \leq C \, \| (\cL_0 + 1/2) \langle x , v \rangle \|_{L^2(\Gamma_0)} \leq C \, \e^{ -\frac{ R^2}{8}}  \, .
\end{align}
Rotate $\RR^N$ so that $\cV$ is the span of
$\partial_{k+2} , \dots , \partial_{n+1}$ and $\cV_0$ is the span of $\partial_{n+2} , \dots , \partial_N$, so that
\begin{align}
	\| 2k - \sum_{i=1}^{k+1} x_i^2 \|_{C^2(\Gamma_0)} &\leq C \, \e^{ -\frac{ R^2}{8}}  \, , 	\label{e:jnplus2p}\\
	\| x_j \|_{C^2(\Gamma_0)} &\leq C \, \e^{ -\frac{ R^2}{8}}  {\text { for each }} j \geq n+2 \, .	\label{e:jnplus2q}
\end{align}
We will extend these bounds off of $\Gamma_0$ using the almost translation invariance in $\cV$. 

Given a vector $v \in \cV$, let $\Gamma_v$ be the ``level set'' of $\Gamma$ at $v$ given by
\begin{align}
	\Gamma_v = \{ q \in \Gamma \, | \, (x_{k+2} (q) , \dots , x_{n+1}(q)) = v \} \, .
\end{align}
This is consistent with the previous definition of $\Gamma_0$.
We need to extend the bounds \eqr{e:jnplus2p} and \eqr{e:jnplus2q} along $\cV$.  
 We model $B_R \cap \Gamma$ on the product $\Gamma_0 \times (B_R^{\RR^{n-k}} \cap \cV)$.

Define the matrix $J_{ij}$, for $i,j = k+2 , \dots , n+1$,  at each point of $B_R \cap \Gamma$ by
\begin{align}
	J_{ij} = \langle \partial_i^T , \partial^T_j \rangle =  \langle \partial_i - \Pi ( \partial_i) ,   \partial_j - \Pi ( \partial_j)  \rangle \, .
\end{align}
By  \eqr{e:vi1},  $J_{ij}$ is close to the identity matrix $\delta_{ij}$
\begin{align}	\label{e:jacJoj}
	\left| J_{ij} - \delta_{ij} \right| \leq 2 \left|  \Pi (\partial_i) \right| + \left|  \Pi (\partial_j) \right| \leq C \, \e^{ \frac{|q|^2-R^2}{8}}\,    (1 +s^2)  \, .
\end{align}
As long as $C \, \e^{ \frac{|q|^2-R^2}{8}}\,    (1 +s^2) $ is small enough, then $J_{ij}$ has an inverse matrix $J^{ij}$ that is also close to the identity.  For each $i = k+2 , \dots , n+1$, define a vector field  $V^i$ by
\begin{align}
	V^i = \sum_{j=k+2}^{n+1} \, J^{ij} \,  \partial_j \, .
\end{align}
Thus, each $V^i$ is in $\cV$,   $|V^i - \partial^T_i| \leq C \, \e^{ \frac{|q|^2-R^2}{8}}\,    (1 +s^2) $  and  $\langle V^i , \partial^T_j \rangle = \delta_{ij}$.  Integrating $(V^i)^T$ gives a flow on $\Gamma$ that changes $x_i$ at speed one and fixes $x_j$ for $j \in \{ k+2 , \dots , n+1\} \setminus \{ i \}$.  These flows allow us to deform $\Gamma_0$ to $\Gamma_v$ for $v \in B_R \cap \cV$ by changing one coordinate at a time.  We will use the flows
 to extend \eqr{e:jnplus2p} and \eqr{e:jnplus2q}  to $\Gamma_v$ by bounding the derivatives    in the direction of $\cV$. First, 
   if $j \notin \{ k+2, \dots , n+1 \}$, then
\begin{align}
	\left| \nabla_{(V^i)^T} x_j \right| = \left| \langle (V^i)^T , \partial_j \rangle \right| = \left| \langle \Pi \left( V^i \right)  , \partial_j \rangle \right| \leq C \, \e^{ \frac{|q|^2-R^2}{8}}\,    (1 +s^2)  \, .
\end{align}
  Integrating this and using \eqr{e:jnplus2q} gives for $j \geq n+2$ that
  \begin{align}
  	|x_j(q)| \leq C \, \e^{ \frac{|q|^2-R^2}{8}}\,    (1 +s^3)  \, .
  \end{align}
  The extension of \eqr{e:jnplus2p} to $\Gamma_v$ follows similarly.  We conclude that $\Gamma$ is a graph  over the cylinder of some $U$ as long as
  $C \, \e^{ \frac{|q|^2-R^2}{8}}\,    (1 +s^3) $ is fixed small.  The proposition follows since we can absorb polynomial factors into the exponential as long as $R$ is sufficiently large.
 \end{proof}

\section{First variation of geometric quantities}	\label{s:s5}

The goal of the next two sections is to   bound  $\| P \|_{L^1}$ in terms of the distance to a cylinder and  $\phi$.  It will be crucial that this bound is better than quadratic in terms of the distance to a cylinder.   This will finally be achieved in Proposition \ref{l:PLone}.  In this section, we Taylor expand $P$ to second order.  We will   show that the first derivative vanishes, giving  a quadratic bound.  To beat this, we need to understand the second order term.  Unfortunately, the second variation of $P$ does not vanish.   However,   it does vanish in the directions of the Jacobi fields which will give just enough control to beat the quadratic bound.  In this, we have suppressed the distinction between pointwise and integral norms; this is dealt with in the next section.
    
  Let $\Sigma^n \subset \RR^N$ be a smooth closed  submanifold, not necessarily a shrinker,  and 
  \begin{align}
  	F: \Sigma \times \RR \to \RR^N
\end{align}
 a smooth mapping with $F(p,0) = p$.  This defines a one-parameter family $\Sigma_s = F(\Sigma, s)$ of submanifolds with $\Sigma_0 = \Sigma$.      Let $\Pi (\cdot)= \Pi (p,s)(\cdot)$ denote orthogonal projection onto the normal bundle for $\Sigma_s$ and $W^T = W - \Pi (W)$  denote tangential projection.

 Let $p_i$ be local coordinates  on $\Sigma$ and $\partial_i$ the corresponding coordinate vector fields.  Using $dF$, we  push this forward to a frame $F_i =dF (\partial_i)$ on $\Sigma_s$.  The metric $g_{ij} = g_{ij} (p,s)$ is  
  \begin{align}
  	g_{ij}= \langle F_i , F_j \rangle  \, .
  \end{align}
  We will use $g^{ij}$ to denote the inverse matrix to $g_{ij}$.
 In this section, we will compute the variations of various geometric quantities, including the   linearization of $\phi$:

  \begin{Pro}	\label{c:firstvarPHI}
  Let $V (p) = F_s (p,0)$.  
If $V^T = 0$, then 
   at $s=0$  
   \begin{align}	\label{e:Phi0}
  	  \phi_s =   L \, V -   F_j \, g^{ij} \, \langle    \nabla_{F_i}^{\perp} V , \phi     \rangle  	\, .
  \end{align}
   \end{Pro}

\vskip1mm
The first term on the right in \eqr{e:Phi0} is the normal part, while the second is the tangent part which vanishes when $\phi = 0$ (i.e., at a shrinker).

  \subsection{The derivatives of $\Pi$ and $A$}

   The next lemma computes the derivative $\Pi_s$ of $\Pi$; the first term on the right is the normal part and the second is the tangent part.    
   
  \begin{Lem}	\label{l:diffPiN}
 Given a vector $W \in \RR^N$, we have
 \begin{align}
 	  \Pi_s (W) = - \Pi \, \left( \nabla_{W^T}  \, F_s \right)  - F_j \, g^{ij} \, \langle  \Pi (  \nabla_{F_i} F_s ) ,  W  \rangle \, .
 \end{align}
  \end{Lem}
  
  \begin{proof}
  Since $\Pi$ is an orthogonal projection, $\Pi$ is symmetric and   $\Pi^2 = \Pi$.  Differentiating these gives that $\Pi_s$ is symmetric and
  $\Pi_s  \, \Pi + \Pi \,  \Pi_s = \Pi_s$.
  Multiplying on the right by $\Pi$ and using that $\Pi^2 = \Pi$ gives that
  \begin{align}	\label{e:PiPisPi}
  	\Pi \, \Pi_s  \, \Pi = 0 \, .
  \end{align}

To calculate $\Pi_s (W^T)$, we use that $\Pi (F_i) \equiv 0$ to get
\begin{align}
	\Pi_s (F_i) = - \Pi (F_{is}) =    - \Pi (  \nabla_{F_i} F_s ) \, ,
\end{align}
where the last equality is the chain rule.
By linearity, it follows that 
\begin{align}	\label{e:PisWT}
	\Pi_s  (W^T) = - \Pi \, \left( \nabla_{W^T}  F_s \right) \, .
\end{align}
  The normal part of $ \Pi_s \Pi (W) $ vanishes by \eqr{e:PiPisPi}.  To compute  the tangential  part of $ \Pi_s \Pi (W) $, we use that $\Pi_s$ is symmetric to get
\begin{align}
	\langle F_i , \Pi_s \Pi (W) \rangle = \langle \Pi_s (F_i ) ,  \Pi (W) \rangle = - \langle  \Pi (  \nabla_{F_i} F_s ) , \Pi (W) \rangle  = - \langle  \Pi (  \nabla_{F_i} F_s ) ,  W  \rangle \, .
\end{align}
It follows that
\begin{align}	\label{e:PisPiW}
	 \Pi_s \Pi (W) = ( \Pi_s \Pi (W) )^T = - F_j \, g^{ij} \, \langle  \Pi (  \nabla_{F_i} F_s ) ,  W  \rangle \, .
\end{align}
The lemma follows by combining \eqr{e:PisWT} and \eqr{e:PisPiW}.
  \end{proof}

 \begin{Cor}	\label{c:XprimeN}
  We have that  $  \left( \Pi (F) \right)_s  = \Pi (F_s) -  \Pi \, \left( \nabla_{F^T}  \, F_s \right)  - F_j \, g^{ij} \, \langle  \Pi (  \nabla_{F_i} F_s ) ,  F  \rangle$.
  \end{Cor}
  
  \begin{proof}
 This follows from Lemma \ref{l:diffPiN} since  	 $\left( \Pi (F) \right)_s =  \Pi (F_s)  + \Pi_s (F)  $.  \end{proof}

\begin{Lem}	\label{c:atzero}
Let $V (p) = F_s (p,0)$.  
If $V^T = 0$, then at $s=0$
\begin{align}
	-\bH_s &= \Delta^{\perp} V + g^{ik} A^V_{km}  g^{mj} \, A_{ij}    + F_k \, g^{mk} \, \langle     \nabla_{F_m} V   ,  \bH  \rangle \, .
\end{align}
\end{Lem}

\begin{proof}
Since the mean curvature $\bH$ of $\Sigma_s$ is minus the trace of $A$, we have
  \begin{align}
  	-\bH = g^{ij} \, \Pi \left( \nabla_{F_i} F_j \right) = g^{ij} \, \Pi \left( F_{ij} \right) \, .
  \end{align}
  Differentiating   gives
  \begin{align}	\label{e:PiHprimeN}
  	- \bH_s  &=   g^{ij} \,   \Pi_s \left( F_{ij} \right) + (g^{ij})_s \, \Pi \left( F_{ij} \right) + g^{ij} \,  \Pi \left( F_{sij} \right)   \, .
  \end{align}
Since $V$ is normal, we have
\begin{align}
	(g_{ij})_s = \langle F_{si} , F_j \rangle + \langle F_i , F_{sj} \rangle = - 2 \, \langle V , F_{ij} \rangle = - 2 A^V_{ij} \, .
\end{align}
It follows that
\begin{align}	\label{e:gijup}
	(g^{ij})_s = - g^{ik} (g_{km})_s g^{mj} = 2 g^{ik} A^V_{km}  g^{mj} \, .
\end{align}
Using this in the second term in \eqr{e:PiHprimeN} gives
\begin{align}	\label{e:qqq1}
	(g^{ij})_s \, \Pi \left( F_{ij} \right) = 2 g^{ik} A^V_{km}  g^{mj} \, A_{ij}  \, .
\end{align}
  Using Lemma \ref{l:diffPiN} on the first term in \eqr{e:PiHprimeN} gives
 \begin{align}	\label{e:qqq2}
 	  g^{ij} \, \Pi_s (F_{ij}) &= - g^{ij} \, \Pi \, \left( \nabla_{F_{ij}^T}  \, F_s \right)  - g^{ij} \,  F_k \, g^{mk} \, \langle  \Pi (  \nabla_{F_m} F_s ) ,  F_{ij}  \rangle \notag \\
	  &= - g^{ij} \,   \nabla_{F_{ij}^T}^{\perp}  \, V    + F_k \, g^{mk} \, \langle     \nabla_{F_m} V   ,  \bH  \rangle
	  \, .
 \end{align}
We rewrite the last term  in \eqr{e:PiHprimeN} as
\begin{align}
	 g^{ij} \,  \Pi \left( F_{sij} \right) &= g^{ij} \,  \Pi \left( \nabla_{F_i} \nabla_{F_j} V  \right) = g^{ij} \,   \nabla_{F_i}^{\perp} \nabla_{F_j}^{\perp} V   + g^{ij} \,   \nabla_{F_i}^{\perp} \nabla_{F_j}^T V   \, .
\end{align} 
To expand $ \nabla_{F_j}^T V$ in terms of the $F_k$'s, take the inner product
\begin{align}	\label{e:5p21}
	\langle F_k , \nabla_{F_j}^T V \rangle = \langle F_k , \nabla_{F_j}  V \rangle = - \langle F_{jk} ,   V \rangle \, , 
\end{align}
so we see that $\nabla_{F_j}^T V =  - F_{\ell} \, g^{\ell k} \langle F_{jk} ,   V \rangle$.  Plugging this in gives
\begin{align}	\label{e:qqq3}
	 g^{ij} \,  \Pi \left( F_{sij} \right)   &= g^{ij} \,   \nabla_{F_i}^{\perp} \nabla_{F_j}^{\perp} V   - g^{ij} \,   \nabla_{F_i}^{\perp} \left( F_{\ell} \, g^{\ell k} \langle F_{jk} ,   V \rangle \right)  \notag \\
	 &= g^{ij} \,   \nabla_{F_i}^{\perp} \nabla_{F_j}^{\perp} V   - g^{ij} \,   A_{i\ell} \, g^{\ell k} A^V_{jk}   \, .
\end{align} 
Inserting \eqr{e:qqq1}, \eqr{e:qqq2}   and \eqr{e:qqq3}    in \eqr{e:PiHprimeN} gives
\begin{align}
	-\bH_s &= 2 g^{ik} A^V_{km}  g^{mj} \, A_{ij} - g^{ij} \,   \nabla_{F_{ij}^T}^{\perp}  \, V    + F_k \, g^{mk} \, \langle     \nabla_{F_m} V   ,  \bH  \rangle +g^{ij} \,   \nabla_{F_i}^{\perp} \nabla_{F_j}^{\perp} V   - g^{ij} \,   A_{i\ell} \, g^{\ell k} A^V_{jk}   \notag \\
	&=   g^{ik} A^V_{km}  g^{mj} \, A_{ij}    + F_k \, g^{mk} \, \langle     \nabla_{F_m} V   ,  \bH  \rangle + \left( g^{ij} \,   \nabla_{F_i}^{\perp} \nabla_{F_j}^{\perp} - g^{ij} \,   \nabla_{F_{ij}^T}^{\perp}  \right) \, V     \, .\end{align}
	The last term in brackets is the normal Laplace operator and the lemma follows.
\end{proof}

   \begin{proof}[Proof of Proposition \ref{c:firstvarPHI}]
   Since $\phi = \frac{1}{2} \Pi (F) - \bH$,  Corollary \ref{c:XprimeN} and Lemma \ref{c:atzero}  give at $s=0$ that
   \begin{align}
   	\phi_s &=  \frac{1}{2} \, \left( V  -   \nabla_{x^T}^{\perp}  V  - F_j \, g^{ij} \, \langle    \nabla_{F_i}^{\perp} V  ,  x  \rangle \right) +  \Delta^{\perp} V + g^{ik} A^V_{km}  g^{mj} \, A_{ij}    + F_k \, g^{mk} \, \langle     \nabla_{F_m} V   ,  \bH  \rangle   \notag \\
		&= L \, V - \frac{1}{2} \, F_j \, g^{ij} \, \langle    \nabla_{F_i}^{\perp} V  ,  x  \rangle + F_k \, g^{mk} \, \langle     \nabla_{F_m} V   ,  \bH  \rangle \, .
   \end{align}
      \end{proof}

 For future reference, we also record the derivative of $A_{ij}$.

\begin{Lem}	\label{l:Aij}
Let $V (p) = F_s (p,0)$.  
If $V^T = 0$, then at $s=0$ and at a point where $F_i$'s are orthonormal and $\nabla_{F_i}^T F_j = 0$
\begin{align}
	(A_{ij})_s = - F_k \, \langle   \nabla_{F_k}^{\perp} V ,  A_{ij} \rangle + \nabla_{F_j}^{\perp} \nabla_{F_i}^{\perp}  V -    A^V_{ik}  \, A_{jk} \, .
\end{align}
\end{Lem}

\begin{proof}
Differentiating $A_{ij} = \Pi (F_{ij} )$,  using  $F_{ij}^T=0$, and then using Lemma \ref{l:diffPiN} gives
\begin{align}	\label{e:uselama}
	(A_{ij})_s = \Pi_s (F_{ij}) + \Pi (V_{ij}) = \Pi_s (A_{ij}) + V_{ij}^{\perp} =  V_{ij}^{\perp}  - F_k \, \langle   \nabla_{F_k}^{\perp} V ,  A_{ij} \rangle\, .
\end{align}
Since $V_i^T = - A^V_{ik} F_k$ by \eqr{e:5p21}, expanding $V_{ij}^{\perp}$ gives
\begin{align}
	V_{ij}^{\perp} & = \nabla_{F_j}^{\perp} \nabla_{F_i} V = \nabla_{F_j}^{\perp} \nabla_{F_i}^{\perp}  V - \nabla_{F_j}^{\perp} \left(   A^V_{ik} F_k \right) 
	=  \nabla_{F_j}^{\perp} \nabla_{F_i}^{\perp}  V -    A^V_{ik}  \, A_{jk} \, .
\end{align}
Using this in \eqr{e:uselama} gives the lemma.
\end{proof}

\subsection{The function $P$ on graphs over the cylinder}

In the next lemma, $P_U$ denotes the quantity $P$ on the graph of $U$ over the cylinder $\SS^k_{\sqrt{2k}} \times \RR^{n-k}$.

\begin{Lem}	\label{l:geeral}
There exist $C$,  $\delta > 0$ and a function $\cP$ so that if $\| U \|_{C^2} \leq \delta$ is any vector field (not necessarily normal), then $P_{U} = \cP (p,U, \nabla U , \nabla^2 U)$ where the function $\cP$ satisfies
\begin{align}	\label{e:PC3bound}
	\| \cP \|_{C^3} \leq C \, (1+ |p|^2) \, .
\end{align}
Moreover, if $U(s)$ is any one parameter family of vector fields with $U(0)=0$ so that $U'(0)$ is tangential, then 
\begin{align}	\label{e:tangential0}
	\frac{d}{ds} \big|_{s=0} \, P_{U(s)} = 0 \, .
\end{align}
\end{Lem}

\begin{proof}  
Each term in the definition  \eqr{e:definePhere} of $P$ is a function of the normal projection $\Pi$, the second fundamental form $A$, the induced metric $g$, and (on the second line of \eqr{e:definePhere}) the position vector $x$.  Each of these objects is a smooth function of $ (p,U, \nabla U , \nabla^2 U)$.  All but the position vector are bounded uniformly near the cylinder; the position vector comes in quadratically, so we get the bound \eqr{e:PC3bound}.

Suppose that $U(s)$ is a  one parameter family of vector fields with $U(0)=0$ so that $U'(0)$ is tangential.  Integrate $U'(0)$ (locally) to get 
  a one-parameter family of diffeomorphisms of the cylinder $\Psi (s)$ with $\Psi (0)$ equal to the identity so that the derivative at $0$ equal to $U'(0)$.  Since $P$ vanishes identically on the cylinder, we get \eqr{e:tangential0}.
\end{proof}

\subsection{The first variation of $P$}

We specialize   to the cylinder $\SS^k_{\sqrt{2k}} \times \RR^{n-k}$ with variations of the form $F_s(p,0) =  V(p)$ where $V = u \, \bN + u^{\alpha} \, \partial_{z_{\alpha}}$ is   normal.  Here, as in Section 
\ref{s:sectionJ}, $y_i$ are coordinates on the axis $\RR^{n-k}$ and $z_{\alpha}$ are  coordinates orthogonal to the cylinder.
 In this case,   
\begin{align}
	A_{ij} = - \frac{\bN}{\sqrt{2k}}\, \cg_{ij} \, , 
\end{align}
where $\cg_{ij}$ is the metric on the spherical factor $\SS^k_{\sqrt{2k}}$.

\vskip1mm
The next lemma shows that  $P'(0)=0$  for any normal variation.  It follows from this and \eqr{e:tangential0}  that     $  P'(0) = 0$  on $\SS^k_{\sqrt{2k}} \times \RR^{n-k}$
for all variations.

\begin{Lem}	\label{l:Pprime}
We have $P' (0)=0$.
\end{Lem}

 Before proving Lemma \ref{l:Pprime}, we observe that the Hessian of $\cP$ vanishes on Jacobi fields.  To keep the notation concise, given a vector field $V$, let $\Hess_{\cP} (V , V)$ denote 
 \begin{align}
 	\Hess_{\cP} (V , V) = \partial^2 \cP \left( [V , \nabla V , \nabla^2 V] ,  [V , \nabla V , \nabla^2 V] \right) \, .
 \end{align}
 
\begin{Cor}	\label{c:2prime}
If   $L\, V=0$ and $\| V \|_{W^{2,2}} < \infty$, then $\Hess_{\cP} (V ,V) =0$ at $\SS^k_{\sqrt{2k}} \times \RR^{n-k}$.
\end{Cor}

\begin{proof}
By Corollary \ref{c:jacobi},  there are constants $b_{ij}$ and a rotation vector field $\bar{V}$ so that
\begin{align}	
	V = \bar{V}^{\perp} + \left\{  \sum_{i \leq j} b_{ij} (y_i y_j - 2 \, \delta_{ij}) \right\} \, \bN 	\, .
\end{align}
Let $\cR(s)$ be the one-parameter family of rotations generated by $\bar{V}$ with $\cR (0)$ equal to the identity.  Next, let $G(s)$ be the variation of $\Sigma$ generated by $ \left\{  \sum_{i \leq j} b_{ij} (y_i y_j - 2 \, \delta_{ij}) \right\} \, \bN$.  Note that $G(s)$ is contained in $\RR^{n+1}$ and, thus, the variation 
\begin{align}
	\bar{F}(s) = \cR(s) \left[ G(s) \right] 
\end{align}
is a hypersurface in the affine space $G(s)[ \RR^{n+1}]$.  In particular, $P(\bar{F}(s)) \equiv 0$ for all $s$ by Lemma \ref{l:PzeroHyper}.
It follows that
\begin{align}
	0 = \frac{d^2}{ds^2} \big|_{s=0} P(\bar{F}(s)) = \langle \nabla \cP , \bar{F}_{ss} \rangle + \Hess_{\cP} (\bar{F}_s , \bar{F}_s) =  \Hess_{\cP} (\bar{F}_s , \bar{F}_s) \, ,
\end{align}
where the last equality used that $\nabla \cP = 0$ at $0$ by Lemma \ref{l:Pprime}.
Observe that tangential variations do not change $\Hess_{\cP}$ at $s=0$ (since $\cP$ and $\nabla \cP$ vanish there) and, moreover,
\begin{align}
	\bar{F}_s^{\perp}   = \left( \cR_s (0) (F) \right)^{\perp} + G_s^{\perp} = V \, ,
\end{align}
so we conclude that $\Hess_{\cP} (V ,V) =0$.  
\end{proof}

\begin{proof}[Proof of Lemma \ref{l:Pprime}]
We will do the calculations in coordinates at a point where the $F_i$'s are orthonormal and $\nabla_{F_i}^T F_j = 0$.  To keep notation short, we will let primes denote $s$ derivatives.
Since $\bN$ and $\partial_{z_{\alpha}}$ are parallel and (pointwise) orthonormal, we see that 
\begin{align}
	\nabla_{F_k}^{\perp} V &= u_k \, \bN + u^{\alpha}_k \, \partial_{z_{\alpha}} \, , \\
	 \nabla_{F_j}^{\perp} \nabla_{F_i}^{\perp}  V &= u_{ij} \, \bN + u^{\alpha}_{ij} \, \partial_{z_{\alpha}} \, .
\end{align}
Using these, Lemma \ref{l:Aij} and $A_{ij} = - \frac{\bN}{\sqrt{2k}}\, \cg_{ij}$, we get 
\begin{align}	\label{e:onemoreAijp}
	A_{ij}' &= - F_k \, \langle   \nabla_{F_k}^{\perp} V ,  A_{ij} \rangle + \nabla_{F_j}^{\perp} \nabla_{F_i}^{\perp}  V -    A^V_{ik}  \, A_{jk}  
	=  \frac{\cg_{ij}\nabla u }{\sqrt{2k}}  +  \left( u_{ij}  -    \frac{u\,  \cg_{ij}}{2k} \right)   \bN + u^{\alpha}_{ij} \, \partial_{z_{\alpha}}  \, .
\end{align}
We conclude that
  \begin{align}	\label{e:conven1}
 	\langle A_{ij} , A_{m\ell} \rangle &= \frac{1}{2k} \, \cg_{ij} \cg_{m\ell} 
	\, , \\
 	\langle A_{ij}' , A_{m\ell} \rangle &= - \langle    A_{ij}' , \bN \rangle \, \frac{1}{\sqrt{2k}} \cg_{m\ell}  = - \left( u_{ij}  -    \frac{u}{2k}\,  \cg_{ij} \right) \frac{1}{\sqrt{2k}} \cg_{m\ell} \, . \label{e:conven2}
 \end{align}
 
Equation \eqr{e:gijup} and  $A_{ij} = - \frac{\bN}{\sqrt{2k}}\, \cg_{ij}$ give that
\begin{align}	\label {e:gijupHere}
	(g^{ij})' =2 A^V_{ij}  = -\frac{2 \, u \, \cg_{ij}}{\sqrt{2k}} \, .
\end{align}
Note that $|\bH| = \frac{\sqrt{k}}{\sqrt{2}}$.
Lemma \ref{c:atzero} gives that
\begin{align}
	-\bH' 
	%&= \left(  -\frac{2 \, u \, \cg_{ij}}{\sqrt{2k}} \right) A_{ij} + A_{ii}' = 
	% \nabla u \, \frac{\cg_{ii}}{\sqrt{2k}}  +  \left( u_{ii}  -    \frac{u}{2k}\,  \cg_{ii} \right) \, \bN + u^{\alpha}_{ii} \, \partial_{z_{\alpha}} 
	%  +  \bN  \cg_{ij}\left(  \frac{ u \, \cg_{ij}}{k} \right)  \notag \\
	 &= \frac{\sqrt{k}}{\sqrt{2}} \, \nabla u + \left( \Delta u   + \frac{u}{2} \right) \, \bN + (\Delta u^{\alpha})  \, \partial_{z_{\alpha}} 
	 \, .
\end{align}
We compute the derivatives of $|\bH|$ and $\bN = \frac{\bH}{|\bH|}$
\begin{align}
	|\bH|' &= \frac{ (|\bH|^2)'}{2|\bH|} = \langle \bN , \bH' \rangle = -\Delta u - \frac{u}{2} \, , \\
	\bN' &= \frac{\bH'}{|\bH|} - \bH \frac{|\bH|'}{|\bH|^2} =  - \nabla u   -\frac{\sqrt{2}}{\sqrt{k}} \, (\Delta u^{\alpha})  \, \partial_{z_{\alpha}}  \, .	\label{e:lastbNp}
\end{align}
We now start differentiating the parts of $P$ using \eqr{e:conven1}--\eqr{e:lastbNp} and the symmetry of $A$ and $g$.  First, we have
\begin{align}
	(|A|^2)' &= \left( g^{im} g^{j\ell} \langle A_{ij} , A_{m\ell} \rangle  \right)' = 2 \langle A_{ij}' , A_{ij} \rangle + 2\langle A_{ij} , A_{mj} \rangle (g^{im})'
	  \notag \\
	&= -     \left( u_{ij}  -    \frac{u}{2k}\,  \cg_{ij} \right)     \frac{\sqrt{2}}{\sqrt{k}}\, \cg_{ij}  +  \frac{1}{k} \, \cg_{im} 
	\left( -\frac{2 \, u \, \cg_{im}}{\sqrt{2k}} \right)   \, .
\end{align}
Simplifying this gives
\begin{align}	\label{e:Asqr}
	   \frac{\sqrt{k}}{\sqrt{2}} \, (|A|^2)' & = -     \left( u_{ij}  \cg_{ij}   -    \frac{u}{2}  \right)       -u = -     \left( u_{ij}  \cg_{ij}   +    \frac{u}{2}  \right)
	   = -     \left( \Delta_{\theta} +    \frac{1}{2}  \right)u
	  \, .
\end{align}
Next, we turn to $|A^{\bN}|^2 = \langle A_{ij} , \bN \rangle \langle A_{\ell m} , \bN \rangle g^{i\ell} g^{jm}$.  We get that
\begin{align}
	\langle A_{ij} , \bN \rangle' &= \langle A_{ij}' , \bN \rangle + \langle A_{ij} , \bN' \rangle =   \left( u_{ij}  -    \frac{u}{2k}\,  \cg_{ij} \right)  \, .
\end{align}
Thus, we have
\begin{align}
	\left( |A^{\bN}|^2 \right)' &= 2\,  \langle A_{ij} , \bN \rangle' \langle A_{ij} , \bN \rangle  + 2\langle A_{ij} , \bN \rangle \langle A_{i m} , \bN \rangle  ( g^{jm})' \notag \\
	&= -2\, \left( u_{ij}  -    \frac{u}{2k}\,  \cg_{ij} \right) \frac{1}{\sqrt{2k}}\, \cg_{ij}
	  + 2 \left( \frac{1}{\sqrt{2k}}\, \cg_{ij} \right) \left( \frac{1}{\sqrt{2k}}\, \cg_{im} \right)  
	  \left( -\frac{2 \, u \, \cg_{jm}}{\sqrt{2k}}
	  \right) \\
	  &= -  \left( u_{ij} \, \cg_{ij} -    \frac{u}{2}  \right) \frac{\sqrt{2}}{\sqrt{k}}
	   -\frac{\sqrt{2}\, u \,  }{\sqrt{k}}
	  = -    \frac{\sqrt{2}}{\sqrt{k}}  \left( \Delta_{\theta} +    \frac{1}{2}  \right)u
  \notag \, .
\end{align}
Since $|A|^2 = |A^{\bN}|^2 = \frac{1}{2}$ on the cylinder, combining these gives
\begin{align}	\label{e:550q}
	\left( |A|^2 \, |A^{\bN}|^2 \right)' & = \frac{1}{2} \, \left(  ( |A|^2)'    +  (  |A^{\bN}|^2  )'   \right)  =  -    \frac{\sqrt{2}}{\sqrt{k}}  \left( \Delta_{\theta} +    \frac{1}{2}  \right)u \, . 
\end{align}
We turn next to $(A^2)_{ij} = \langle A_{im} , A_{\ell j} \rangle g^{m\ell}$.  We have
\begin{align}	\label{e:A2ijpri}
	(A^2)_{ij}' &= \langle A_{im}' , A_{m j} \rangle  +  \langle A_{im} , A_{m j}' \rangle  +  \langle A_{im} , A_{\ell j} \rangle (g^{m\ell})' \notag \\
	&= -\left( u_{im} \cg_{mj}  -    \frac{u}{2k}\,  \cg_{ij} \right) \frac{1}{\sqrt{2k}}\, 
	- \left( u_{mj}  \cg_{im}  -    \frac{u}{2k}\,  \cg_{ij} \right)  \frac{1}{\sqrt{2k}}\,
	- \frac{1}{{k}}\,   \frac{u}{\sqrt{2k}} \cg_{ij}   \\
	& =  -\left( u_{im} \cg_{mj} + u_{mj}  \cg_{im}    \right) \frac{1}{\sqrt{2k}} 		
	\notag \, .
\end{align}
Note that $A^2 = \frac{\cg}{2k}$ on the cylinder.  Using the symmetry of $A^2$ and $g$, \eqr{e:A2ijpri} gives that
\begin{align}	\label{e:552q}
	\left( \left| A^2 \right|^2 \right)' &= 2 \, (A^2)_{ij}' A^2_{ij} + 2 \, A^2_{ij} A^2_{i\ell} (g^{j\ell})' =
	-  \left( u_{im} \cg_{mj} + u_{mj}  \cg_{im}    \right) \frac{1}{\sqrt{2k}} 	 \frac{\cg_{ij}}{k} -     \frac{ \cg_{j\ell}}{2k^2}   \left( \frac{2 \, u \, \cg_{j\ell}}{\sqrt{2k}} \right) \notag \\
	&=  -\frac{\sqrt{2}}{k\sqrt{k}} \left( \Delta_{\theta} + \frac{1}{2} \right) \, u   \, .
\end{align}
Since  $A(x^T , \cdot ) = 0$ on the cylinder, it follows that  
\begin{align}	\label{e:553q}
	\left( \frac{|A|^2}{4|\bH|^2} \, \left[  \left| A^{\bN} (x^T , \cdot)\right|^2 - \left| A  (x^T , \cdot)\right|^2 \right] \right)' = 
	0 \, .
\end{align}
 Next, using the symmetry of $A$ and $g$,  \eqr{e:conven1} and \eqr{e:conven2} give
  \begin{align}	\label{e:554q}
 	\left( |\langle A_{ij} ,  A_{m\ell} \rangle|^2 \right)' &= 4\, \langle A_{ij}' , A_{m\ell} \rangle \langle A_{ij} , A_{m\ell} \rangle + 
	4\,  \langle A_{ij}, A_{m\ell} \rangle \langle A_{ij} , A_{md} \rangle (g^{\ell d})' \notag \\
	&=  - 4 \left( u_{ij}  -    \frac{u}{2k}\,  \cg_{ij} \right) \frac{1}{\sqrt{2k}} \cg_{m\ell}  \frac{1}{2k} \, \cg_{ij} \cg_{m\ell} 
 	 + 4\,  \frac{1}{2k} \, \cg_{ij} \cg_{m\ell} \,  \frac{1}{2k} \, \cg_{ij} \cg_{md} \, \left(- \frac{2u \cg_{\ell d}}{\sqrt{2k}}   \right) 
	\\
	&= -  \left( \Delta_{\theta} u  -    \frac{u}{2}  \right) \frac{\sqrt{2}}{\sqrt{k}}     
 	 -       \, \left( \frac{2u  }{\sqrt{2k}}   \right) =  -  \left( \Delta_{\theta} u  +    \frac{u}{2}  \right) \frac{\sqrt{2}}{\sqrt{k}}   
	 \notag \, .
 \end{align}
  For the remaining term, we compute
  \begin{align}
  	 & \left( \langle A_{j_1 \ell_1}   , A_{i_1 m_1}  \rangle  \, \langle A_{\ell_2 m_2} , A_{i_2 j_2} \rangle g^{m_1 m_2} g^{j_1 j_2}g^{\ell_1 \ell_2} g^{i_1i_2} \right)'   \notag  \\
  	 & \qquad =
	 4\,  \langle A_{j \ell}   , A_{im}  \rangle  \, \langle A_{\ell m} , A_{ij}' \rangle  +   4\, \langle A_{j \ell}   , A_{im_1}  \rangle  \, \langle A_{\ell m_2} , A_{ij} \rangle (g^{m_1 m_2} )'   \label{e:556q} \\
	 &\qquad = -\frac{2}{k} \,  \cg_{j \ell}    \cg_{im}      \left( u_{ij}  -    \frac{u}{2k}\,  \cg_{ij} \right) \frac{1}{\sqrt{2k}} \cg_{m\ell}   -
	      \frac{\sqrt{2}u  }{k\sqrt{k}}   	 =	-\frac{\sqrt{2}}{k\sqrt{k}} \left( \Delta_{\theta} + \frac{1}{2} \right) \, u   \, . \notag 
  \end{align}
Adding \eqr{e:550q}, \eqr{e:553q}, and twice \eqr{e:556q} and then subtracting  \eqr{e:554q} and twice \eqr{e:552q}  gives  that $P'=0$, completing the proof.
\end{proof}

 \section{Estimates for entire graphs}	\label{s:s6}
 
 In this section, $\Sigma^n = \SS^k_{\sqrt{2k}} \times \RR^{n-k} \subset \RR^N$ is a fixed cylinder and all constants $C$ will be allowed to depend on  $N$.  
 Given a normal vector field $U$ on $\Sigma$, let $\Sigma_U$ denote its graph and let $\phi_U$ and $P_U$ be the quantities $\phi$ and $P$ on $\Sigma_U$.

The main result of this section is the next proposition that gives a  bound for $\| P_U \|_{L^1}$ that is essentially quadratic in $\phi_U$ and better than quadratic in $\| U \|_{L^2}$.  This will be crucial in the improvement step in the next section, where we combine it with Theorem \ref{c:kappa} to  bound   $\nabla \tau$.
 
 \begin{Pro}	\label{l:PLone}
 There exists $C$ and $\bar{\epsilon}>0$,  depending on $N$,  so that if $\epsilon_0 \leq \bar{\epsilon}$, then 
 for any $\kappa \in (0,1]$ there is a constant $C_{\kappa} = C_{\kappa}(\kappa , N)$ so that 
        \begin{align}
        		\| P_{U} \|_{L^1} \leq C_{\kappa} \, \| U \|_{L^2}^3  + C_{\kappa} \, \| \phi_U \|_{L^2}^{ \frac{6}{3+\kappa}}
		+ C \, \| U \|_{L^2} \,  \|  \phi_{U} \|_{L^2}  +  C \,  \|  \phi_{U} \|_{W^{1,2}}^2
		\, .
	\end{align}
 \end{Pro}

 At a formal level, the proposition   follows from Taylor expansion and   the properties of  the first two derivatives of $P$ in the previous section.  However, Taylor expansion naturally leads to Gaussian $L^1$ bounds on powers of $U$, which are not in general bounded by powers of   $\| U \|_{L^2}$.   We will give some preliminary  results that are used to bridge this gap.
 
 \vskip1mm
Let $J$ denote the $L^2$ orthogonal projection of $U$ to the Jacobi fields of Proposition \ref{l:jacobi}
and set $h = U - J$.  
  It will be convenient to define $\| \cdot \|_2$ on a normal vector field $U$ by 
 \begin{align}
 	\| U \|_2 =   \left\| |U|^2 + |\nabla U|^2 + |\nabla_{\RR^{n-k}} |\nabla U| |^2 + \frac{ \left| \Hess_U \right|^2}{1+|x|}    \right\|_{L^2} \, .
 \end{align}
 Note that this is essentially quadratic in $U$, but is not necessarily  bounded by $\| U \|_{W^{2,2}}^2$.

  \subsection{An intermediate gradient inequality}
 
 Define $F(U)$ by
 \begin{align}
 	F(U) = F(\Sigma_U) - F(\Sigma) \, .
 \end{align}
 Loosely speaking, $\phi_U$ is the gradient of $F$ and the next proposition is a gradient inequality.  
 
 \begin{Pro}	\label{p:gradLoj}
 There exist $C$ and $\epsilon_0 > 0$ so that if $U$ is a compactly supported normal vector field on $\Sigma$ with $\| U \|_{C^2} < \epsilon_0$, then
 \begin{align}
 \| h \|_{W^{2,2}}  &\leq C  \,\left(  \| U \|_{L^2}^2 + \|  \phi_{U} \|_{L^2} \right)     \, ,  \label{e:spect}  \\
 	|F(U)| &\leq C \, \| \phi_U \|_{L^2} \, \| U \|_{L^2} + C \, \| U \|_{L^2}^3 \, .
 \end{align}
 \end{Pro}
 
 We will use the next two lemmas to prove the proposition. The first of these   bounds $L$ 
from $W^{2,2}$ to $L^2$ and shows that   $\| h \|_2$ is much smaller than $ \| h \|_{ W^{2,2} } $.
 
 \begin{Lem}	\label{l:poinc}
 There exists $C $ so that 
 \begin{align}
 	\| L \, U \|_{L^2} &\leq C \, \| U \|_{W^{2,2}} \, , \\
	\| J \|_2 &\leq C \, \| J \|_{L^2}^2 \leq C \, \| U \|_{L^2}^2 \, ,  \label{e:lpoincJ} \\
	\| h \|_2  &\leq C \, \| U \|_{C^2} \, \| h \|_{ W^{2,2} } \, . \label{e:lpoinc2}
 \end{align}
 \end{Lem}
 
 \begin{proof}
 The operator $L$ on $\Sigma$ becomes
$
	L = \cL + \frac{1}{2} + \frac{1}{2} \, \Pi_{\bN}$, where $\Pi_{\bN}$ is orthogonal projection onto the principal normal $\bN$.  Therefore, we have
\begin{align}
	\| L \, U \|_{L^2} &\leq  \| \Delta U \|_{L^2} + \frac{1}{2} \| \nabla_{x^T} U \|_{L^2} + \| U \|_{L^2} \leq  n \,  \| \Hess_U \|_{L^2} + \frac{1}{2} \| |x^T| \, |\nabla  U| \|_{L^2} + \| U \|_{L^2} \, . \notag
\end{align}
The first claim follows from this since lemma $3.4$ in \cite{CM3} gives $\| |x| \, |\nabla  U| \|_{L^2} \leq C \| |\nabla U| \|_{W^{1,2}}$.

The key for the other claims is  that Lemma \ref{l:Jbound} gives
\begin{align}	\label{e:myJbound}
	|J|   \leq C \, (1+ |x|^2) \, \| J \|_{L^2 } &\leq C \, (1+ |x|^2) \, \| J \|_{L^2} \, , \\
	  |\nabla  J| + |\Hess_J| &\leq C \, (1+ |x|) \, \| J \|_{L^2}   \,  , \\
	    \left| \Hess_J ( \cdot , \RR^{n-k}) \right| &\leq C \,  \| J \|_{L^2 }  \, . \label{e:onemorec1}
\end{align}
The first inequality in \eqr{e:lpoincJ}  follows  from this since the polynomial factors are bounded uniformly in $L^2$; the second inequality follows since $\| U \|_{L^2}^2 = \| J \|_{L^2}^2 + \|h \|_{L^2}^2$.

To prove \eqr{e:lpoinc2}, we will show that the $L^2$ norms of  $|h|^2 , |\nabla h|^2$, $ |\nabla_{\RR^{n-k}} |\nabla h| |^2$ and $\frac{ \left| \Hess_h \right|^2}{1+|x|} $ are each bounded by
$C \, \| U \|_{C^2} \, \| h \|_{ W^{2,2} }$.  To handle the first, we use that
\begin{align}
	|h|^2 \leq |h| \, \left( |U|  + |J|  \right) \leq  |U|_{C^0} \, |h|   + C \, |h| \, \| U\|_{L^2} (1 + |x|^2) \, , 
\end{align}
 where the last inequality used \eqr{e:myJbound}.  Taking the $L^2$ norm and using lemma $3.4$ from \cite{CM3} on the last term (twice) gives
 \begin{align}
 	\| |h|^2 \|_{L^2} &\leq |U|_{C^0} \, \|h\|_{L^2} + C \,  \| U\|_{L^2} \, \| h \,  (1 + |x|^2)\|_{L^2}   \leq  |U|_{C^0} \, \|h\|_{L^2} + C \,  \| U\|_{L^2} \, \| h \|_{W^{2,2}} \, . \notag
 \end{align}
 Arguing similarly gives the corresponding bounds for $\| |\nabla h|^2 \|_{L^2}$ and $\left\| \frac{ \left| \Hess_h \right|^2}{1+|x|} \right\|_{L^2}$; in the second case, we  avoid taking additional derivatives because of the $(1+|x|)$ in the denominator.  Bounding the last term is similar, but also uses the Kato inequality
 \begin{align}
 	\left| \nabla_{\RR^{n-k}} |\nabla h | \right| \leq \left| \nabla_{\RR^{n-k}} \nabla h \right| \leq \left| \nabla_{\RR^{n-k}} \nabla U \right| + \left| \nabla_{\RR^{n-k}} \nabla J \right| \, , 
 \end{align}
 and then uses \eqr{e:onemorec1} to bound the last term.
  \end{proof}
  
We show next   that $\phi_U = L \, h$ up to higher order terms ($ \|U \|_{2}$ is    quadratic in $U$).
 
 \begin{Lem}	\label{l:Taylorforme}
 There exists $C$ so that
% \begin{align}
 	 $ \|  \phi_U - L \, h \|_{L^2}   \leq  C \,  \|U \|_{2} \, .$  %\label{e:taylor2}
% \end{align}
 \end{Lem}
 
 \begin{proof}
 Taylor expansion as in lemma  $4.10$ of \cite{CM3} (see $(4.19)$ in particular) gives that{\footnote{Section $4$ in \cite{CM3} considers variations by functions rather than normal vector fields, but lemma $4.10$ is a general calculus fact that goes through for vector fields.  The calculation  for the linearization of $\phi$  in \cite{CM3} is replaced here by Proposition \ref{c:firstvarPHI}.}}
 \begin{align}	\label{:taylor1}
 	|\phi_U - L \, U | &\leq C \, (1+|x|) (|U|^2 + |\nabla U|^2) + C \, (|U|  + |\nabla U| )\left| \Hess_U \right| \notag  \\
		&\leq C \, (1+|x|) (|U|^2 + |\nabla U|^2) + \frac{C}{(1+|x|)} \,  \left| \Hess_U \right|^2 \, .
 \end{align}
%The second inequality used the absorbing inequality $2 \, (|U|  + |\nabla U| )\left| \Hess_U \right| \leq (1+|x|) (|U|^2 + |\nabla U|^2) + \frac{1}{(1+|x|)} \,  \left| \Hess_U \right|^2$.
Applying lemma $3.4$ from \cite{CM3} gives
  \begin{align}	\label{:taylor2a}
 	\|\phi_U - L \, U \|_{L^2} &\leq   C \, \left\| |U|^2 + |\nabla U|^2 + |\nabla_{\RR^{n-k}} |\nabla U| |^2 + \frac{ \left| \Hess_U \right|^2}{1+|x|}    \right\|_{L^2} \equiv  C \, \| U \|_2 \, .
 \end{align}
 \end{proof}

 \begin{proof}[Proof of Proposition \ref{p:gradLoj}]  
The squared triangle inequality and \eqr{e:lpoincJ}  in Lemma \ref{l:poinc} give
 \begin{align}	\label{e:combinthis}
 	\| U \|_2 \leq 2\, \| J \|_2 + 2 \, \| h \|_2 \leq C \, \| U \|_{L^2}^2 + 2 \, \| h \|_2 \, .
 \end{align}
 Combining \eqr{e:combinthis} with the third claim in Lemma \ref{l:poinc}, and noting that $\| U \|_{C^2} \leq \epsilon_0$, gives 
  \begin{align}	\label{e:U2bd}
 	\| U \|_2 \leq  C \, \| U \|_{L^2}^2 + C \,\epsilon_0 \, \| h \|_{ W^{2,2} } \, .
 \end{align}
Combine Corollary \ref{c:eff} with Lemma \ref{l:Taylorforme} and \eqr{e:U2bd}  to get
 \begin{align}	\label{e:U2bdd}
 		 \| h \|_{W^{2,2}} &\leq C_1 \,  \| L h \|_{L^2}   \leq C_1  \| \phi_U \|_{L^2}  + C \, C_1  \|U \|_{2}  \notag \\
		&\leq C  \| \phi_U \|_{L^2}  + C\, \| U \|_{L^2}^2 +   C \, \epsilon_0 \| h \|_{W^{2,2}}    \, .  
 \end{align}
 Taking $\epsilon_0$ small (depending on $C$ which depends on $n$), we can absorb the last term to get
  \begin{align}
 		 \| h \|_{W^{2,2}} &\leq   C  \| \phi_U \|_{L^2}  + C\, \| U \|_{L^2}^2     \, . \label{e:taylor4}
 \end{align}
Using this in \eqr{e:U2bd} and \eqr{e:U2bdd} gives
   \begin{align}	\label{e:U2bdA}
 	\| U \|_2 +
	 \| L h \|_{L^2}   \leq C \, \| U \|_{L^2}^2 + C  \| \phi_U \|_{L^2}  \,  . 
 \end{align}

 Now consider  the one-parameter family of graphs of $s\, U$ for $s \in [0,1]$.  By the first variation,  
\begin{align}
	\frac{d}{ds} \, F(s\, U) = - \int_{\Sigma_{s\, U}} \langle  \phi_{s\,U} , U \rangle \, \e^{-f} \, .
\end{align}
The fact that $U$ is normal implies that  $\e^{-f}$ on $\Sigma_{s\,U}$ is at most $\e^{\sup |U|^2} \, \e^{-f}$ on $\Sigma$.  
Moreover,  it follows that \eqr{:taylor2a} holds for $s\, U$.  
Since the area elements are uniformly equivalent up to $C(|U| +|\nabla U|)$,   the fundamental theorem of calculus (cf. ($4.25$) in \cite{CM3}) gives
 \begin{align}
 	\left| F(U) + \frac{1}{2} \langle U , L \, U \rangle_{L^2} \right|  &\leq C \, \| U \|_{L^2} \, \|  U \|_2  + C \, \| |L \, U| \, |U| \, (|U| +|\nabla U|) \|_{L^2} \notag \\
	&\leq C \, \| U \|_{L^2} \, \|  U \|_2  + C \, \| L \, U \|_{L^2} \, \| U \|_2 \, .    \label{e:taylor0}
 \end{align}
Consequently, 
  \eqr{e:taylor0} and the bound  \eqr{e:U2bdA} for $\| L \, h \|_{L^2} + \| U \|_2$  give 
 \begin{align}
 	|F(U)|    
	&\leq C \, \| U \|_{L^2} \,\left(  \| L h\|_{L^2} +  \|U \|_{2} \right) \leq C \, \| U \|_{L^2}    \| \phi_U \|_{L^2}    +  C \, \| U \|_{L^2}^{3}   \, .
\end{align}
 \end{proof}
 
    \subsection{Distance to a Jacobi field}
       
      In this subsection, it will be useful to have the following notation for the pointwise $C^2$ norm of a normal vector field $V$:
 \begin{align}
 	|V|_2 \equiv |V| + |\nabla V| +|\nabla^2 V| \, .
 \end{align}
       The next lemma bounds the distance from $U$ to the Jacobi field $J$.  Recall that $h=U-J$.

       \begin{Lem}	\label{l:revHolder}
       There exists $C$ and ${\epsilon}_0>0$,  depending on $N$,  so that if $\| U \|_{C^2} \leq \epsilon_0$, then 
       \begin{align}
       			\int (1+|x|^2)|h|_2^2 \, \e^{-f} &\leq    C \,  \| U \|_{L^2}^4 + C \, \| \phi_U \|_{W^{1,2}}^2   \, .  \label{e:e628}
	\end{align}
	Moreover, given $\kappa \in (0,1]$, there exists $C_{\kappa} = C_{\kappa} (\kappa , N)$ so that
	\begin{align}
		\| (1+|x|^6) |U|_2^3 \|_{L^1} &\leq C_{\kappa} \, \left\{ \| U \|_{L^2}^3 +   \| \phi_{U} \|_{L^2}^{ \frac{6}{3+\kappa} } \right\}  \, .
	\end{align}
       \end{Lem}
       
      \begin{proof}    
    We will need the following simple observation: If $u$ is a function on the cylinder with $\int |u|_2^2(1+|x|^2) \e^{-f} < \infty$, then taking the divergence of $(1+|x|^2) \nabla |\nabla u|^2\, \e^{-f}$ gives
\begin{align}
	\left| \int (1+|x|^2) \, \cL \, |\nabla u|^2 \, \e^{-f} \right| =  2\, \left|  \int \langle x^T , \nabla |\nabla u|^2 \rangle \, \e^{-f} \right| \leq 4 \, \int |x| \, |\nabla u| |\Hess_u| \, \e^{-f}  \, .
\end{align}
The drift Bochner formula{\footnote{The general Bochner formula for $\cL$ is $\frac{1}{2} \, \cL |\nabla u|^2 = |\Hess_u|^2 + \langle \nabla \cL u , \nabla u\rangle
+ \langle (\Ric+ \Hess_f )(\nabla u) , \nabla u \rangle$. On the cylinder, $\Ric + \Hess_f$ is $\frac{1}{2}$ times the identity.}} and integration by parts
leads to the bound
\begin{align}
	\int (1+|x|^2) & \left( |\nabla u|^2 +  2\, \left| \Hess_u \right|^2 \right) \, \e^{-f} \leq  4  \, \int |x| \, |\nabla u| |\Hess_u| \, \e^{-f}  -2 \, 
	\int (1+|x|^2) \, \langle \nabla \cL u , \nabla u \rangle \, \e^{-f}     \notag \\
	& \leq 4\, \int |x| \, |\nabla u| \, |\Hess_u|   \, \e^{-f} +  2\, \int (1+|x|^2)     (\cL u)^2   \e^{-f} +  4\int |x| \, |\nabla u| \,   |\cL u| \, \e^{-f}\, .  \notag
\end{align}
Using  absorbing inequalities on the first and third terms, we conclude that
\begin{align}
	\int (1+|x|^2) \left( |\nabla u|^2 + \left| \Hess_u \right|^2 \right) \, \e^{-f} &\leq  4  \int (1+|x|^2)    (\cL u)^2   \e^{-f} +  6\, \int   |\nabla u|^2  \e^{-f}  \, .
\end{align}
The normal bundle to the cylinder is trivial, so we get similarly for   a normal vector field $V$
\begin{align}
	\int (1+|x|^2) \left( |\nabla V |^2 + \left| \Hess_V \right|^2 \right) \, \e^{-f} &\leq   4 \, \int  (1+|x|^2)   |\cL \, V|^2  \e^{-f} + 6 \, \| V \|_{W^{1,2}}^2  \, ,
\end{align}
Since $\cL \, V$ differs from $L \, V$ by curvature terms applied to $V$ and $A$ is bounded, we  get
\begin{align}	\label{e:wtdW22}
	\int (1+|x|^2)|V|_2^2 \, \e^{-f} &\leq   C \, \| V \|_{W^{1,2}}^2 + C \, \int  (1+|x|^2)   |L \, V|^2  \e^{-f}  \, .
\end{align}
Equation \eqr{:taylor1} in the proof of  Lemma \ref{l:Taylorforme}
gives that
\begin{align}
	 \int  (1+|x|^2)   |L U|^2  \e^{-f} \leq  2\int  (1+|x|^2)   |\phi_U|^2  \e^{-f} + C \,  \int   \left\{  (1+|x|^4) |U|_1^4  + |\Hess_U|^4     \right\}  \e^{-f}  \notag \, .
\end{align}
Combining this with \eqr{e:wtdW22} with $V=h$ and using lemma $3.4$ from \cite{CM3} again gives
\begin{align}	\label{e:wtdW22a}
	\int (1+|x|^2)|h|_2^2 \, \e^{-f} &\leq   C \, \| h \|_{W^{2,2}}^2 + C \, \| \phi_U \|_{W^{1,2}}^2  + C \,  \int    (1+|x|^2) |U|_2^4   \,  \e^{-f}  \notag \\
	&\leq C \,  \| U \|_{L^2}^4 + C \, \| \phi_U \|_{W^{1,2}}^2  + C \,  \int  (1+|x|^2) |U|_2^4  \,  \e^{-f}
	 \, ,
\end{align}
where the last inequality used \eqr{e:spect}.  To bound the last term, we use first  $|U|_2 \leq |J|_2 + |h|_2$ and then  the absorbing inequality to get
\begin{align}	 \label{e:absorblast}
	   (1+|x|^2) |U|_2^4   &\leq 2 (1+|x|^2) |U|_2^2 \, |h|_2^2 + 2 (1+|x|^2) |U|_2^2 \, |J|_2^2 \notag \\
	   & \leq 2 (1+|x|^2) |U|_2^2 \, |h|_2^2 
	   + 2 \, (1+|x|^2) |J|_2^4 + \frac{1}{2} \, (1+|x|^2) |U|_2^4
	 \, .
\end{align}
 Lemma \ref{l:Jbound}
gives $C$ so that
\begin{align}	\label{l:JboundA}
	|J|_2 \leq C \, (1+ |x|^2) \, \| J \|_{L^2 (B_{2n})}  \leq C \, (1+ |x|^2) \, \| U \|_{L^2} \,  .
\end{align}
Absorbing the last term in \eqr{e:absorblast}, integrating and using \eqr{l:JboundA} gives
\begin{align}
	 \int  (1+|x|^2) |U|_2^4  \,  \e^{-f} \leq 4 \, \epsilon_0^2 \, \int(1+|x|^2) |h|_2^2 \, \e^{-f} + C \, \| U \|_{L^2}^4 \, .
\end{align}
Using this and the $C^2$ bound on $U$ in \eqr{e:wtdW22a} and choosing $\epsilon_0 > 0$ small enough to absorb the first term on the right above gives \eqr{e:e628}.

Given $\kappa \in ( 0,1]$, we argue as in \eqr{e:absorblast}, but replace the absorbing inequality  with Young's inequality  $ab \leq \delta^p \, \frac{a^p}{p} + \delta^{-q} \, \frac{b^q}{q}
$ (with $a = |U|_2^{1+\kappa}$, $b= |J|_2^2$, $p= \frac{3+\kappa}{1+\kappa}$ and $q=\frac{3+\kappa}{2}$)  to get
\begin{align}	
	    |U|_2^{3+\kappa}   &  \leq 4 \, |U|_2^{1+\kappa} \, |h|_2^2 
	   + c_{\kappa} \,  |J|_2^{3+\kappa}
	 \, .
\end{align}
Integrating this and using \eqr{l:JboundA} to bound the $J$ term and \eqr{e:spect} on the $h$ term gives
\begin{align}	\label{e:6p41}
	\int  |U|_2^{3+\kappa}  \, \e^{-f} \leq  C\,  \| \phi_U \|_{L^2}^2  +   C\, \|U \|_{L^2}^4 + C \, c_{\kappa} \, \| U \|_{L^2}^{3+ \kappa}
	\leq  C\,  \| \phi_U \|_{L^2}^2    + C_{\kappa} \, \| U \|_{L^2}^{3+ \kappa} \, .
\end{align}
The last claim follows from \eqr{e:6p41} and the H\"older inequality
\begin{align}
	\int  |U|_2^{3}(1+|x|^6) \, \e^{-f} &\leq \left( \int |U|_2^{3+\kappa} \, \e^{-f} \right)^{ \frac{3}{3+\kappa}} \, \left( \int (1+|x|^6)^{ \frac{3+\kappa}{\kappa} } \, \e^{-f}
	\right)^{ \frac{\kappa}{3+\kappa}}  	\, .
\end{align}
      \end{proof}

 \begin{proof}[Proof of Proposition \ref{l:PLone}]
 Define $P(s) = P_{sU}$ for  $s \in [0,1]$. By
 Lemmas \ref{l:geeral} and \ref{l:Pprime},
 $P(s) = \cP (x,sU, s\nabla U , s\nabla^2 U)$ where  
 $	\| \cP \|_{C^3} \leq C \, (1+ |x|^2)$,    $\cP(0)=0$, and  $\nabla \cP=0$ at $0$.
 Taylor expansion gives
 \begin{align}	\label{e:e7p4}
 	|P_{U}| & \leq \left| P(0) + P'(0) + \frac{1}{2} \, P''(0) \right| + C \, (1+|x|^6)\, |U|_2^3  =   \frac{1}{2} \,  \left|  P''(0) \right| + 
	C \, (1+|x|^6)\,  |U|_2^3 \, .
 \end{align}
 The Hessian of $\cP$  vanishes in the $(J, \nabla J , \nabla^2 J)$ direction by Corollary \ref{c:2prime}, so
\begin{align}
	\left|  P''(0) \right| &\leq C \, (1+|x|^2) |h|_2 \,|J|_2   + C\,(1+|x|^2)|h|_2^2 \, .
\end{align}
Using this in \eqr{e:e7p4} and integrating
  gives 
\begin{align}	\label{e:PUbound0}
	\| P_{U} \|_{L^1} \leq C \, \|  (1+|x|^2) \, |h|_2^2 \|_{L^1} + C \, \|  (1+|x|^6) \, |U|_2^3 \|_{L^1} + C \, \|  (1+|x|^2) \, |J|_2 \, |h|_2 \|_{L^1} \, .
\end{align}	
Using the Cauchy-Schwarz inequality, Lemma \ref{l:Jbound}, $\| J \|_{L^2} \leq \| U \|_{L^2}$ and  \eqr{e:spect} gives
\begin{align}
	 \|  (1+|x|^2) \, |J|_2 \, |h|_2 \|_{L^1} &\leq  \|  (1+|x|^2) \, |J|_2   \|_{L^2} \,  \|    |h|_2 \|_{L^2} \leq C \, \| J \|_{L^2} \, \| h \|_{W^{2,2}} \notag \\
	 &\leq C \, \| U \|_{L^2} \, \left(  \| U \|_{L^2}^2 + \|  \phi_{U} \|_{L^2} \right) =
	 C\,\| U\|_{L^2}^3+C\,\|U\|_{L^2}\,\|\phi_{U}\|_{L^2}
	  \, .
\end{align}
 Using this in \eqr{e:PUbound0}  and applying Lemma \ref{l:revHolder}  gives for each $\kappa \in (0,1]$
 \begin{align}	 
	\| P_{U} \|_{L^1} \leq  C \,  \| U \|_{L^2}^4 + C \, \| \phi_U \|_{W^{1,2}}^2 + C \, \| U \|_{L^2} \, \| \phi_U \|_{L^2} +
	 C_{\kappa} \, \left\{ \| U \|_{L^2}^3 +   \| \phi_{U} \|_{L^2}^{ \frac{6}{3+\kappa} } \right\} \, ,
\end{align}	
where $C_{\kappa}$ depends on $\kappa $ and $N$.
The proposition follows.
 \end{proof}

\section{Uniqueness of blowups}

In this section,   we will prove uniqueness as outlined in subsection \ref{ss:overview}.   
Throughout, 
$\Sigma_t^n \subset \RR^N$ will be a rescaled MCF with entropy $\lambda (\Sigma_t) \leq \lambda_0$ for some fixed $\lambda_0$ (all constants will be allowed to depend on $\lambda_0$).  In addition, fix some $\beta > 0$.

The key  to uniqueness is the following discrete differential inequality for $F(\Sigma_t)$
(see Definition \ref{d:epsilonclose} for the notion of 
$(\epsilon_1 , R_1 , C^{2,\beta})$-close):

 \begin{Thm}	\label{t:discre}
 Given $n$, $N$,  there exist $K , \bar{R} , \epsilon$ and $\alpha \in (1/3, 1)$ so that if  
  $\Sigma_t$ is $(\epsilon , \bar{R} , C^{2,\beta})$-close to some cylinder $\cC_k$
  for $t \in [T-1 , T+1]$, then 
 \begin{align}	\label{e:7p2J}
 	\left| F (\Sigma_T) - F(\cC_k) \right|^{ 1 +\alpha} \leq K \, \left[ F(\Sigma_{T-1}) - F(\Sigma_{T+1}) \right]  \, .
 \end{align}
 \end{Thm}

\vskip1mm
 The  estimate \eqr{e:7p2J} plays the role of  \eqr{e:LjF}.  Iterating \eqr{e:7p2J}  gives a rate of decay that leads to uniqueness.   This requires  that $\Sigma_t$ remains close to a cylinder forward in time, so that  \eqr{e:7p2J}  continues to hold.
 The next theorem will be used both to prove  \eqr{e:7p2J} and to show that  \eqr{e:7p2J}  applies forward in time.  
 This result
 shows that if $\Sigma_t$ is close to a cylinder in a fixed large ball, then it is close  on a scale    larger than 
the   ``shrinker scale'' $R_T$  defined by
\begin{align}	\label{e:e7p1}
	 \e^{ - \frac{R_T^2}{2} }  = F( \Sigma_{T-1}) - F(\Sigma_{T+1})  
	   \, .
\end{align}

\begin{Thm}	\label{t:shrinkerscale}
Given   $\epsilon_0 > 0$, there exist $R_1$, $\mu > 0$ and $\epsilon_1 > 0$ so that if $\Sigma_t$ is   $(\epsilon_1 , R_1 , C^{2,\beta})$-close to a fixed cylinder for $t \in [T-1 , T+1]$, then:
\begin{enumerate}
\item[(A)]    There is a cylinder $\Gamma$ and compactly supported normal vector field $U$ on $\Gamma$   so that $B_{(1+ \mu)R_T} \cap \Sigma_T$ is contained in the graph of $U$ with  $\| U \|_{C^{2,\beta}} \leq \epsilon_0$, 
\begin{align}	\label{e:7point2}
		\| U \|_{L^2}^2      \leq C_n \, \lambda_0 \, R_T^{n-2} \, \e^{ - \frac{  R_T^2}{4} }   {\text{ and }} 
		 \| \phi_U \|_{L^2}^2     \leq \e^{ - (1+\mu)^2 \frac{  R_T^2}{4} } \, .
\end{align}
\item[(B)] For each $\ell$, there exists $C_{\ell}$ so that $\sup_{ B_{(1+ \mu)R_T} \cap \Sigma_t} \,  | \nabla^{\ell} A| \leq C_{\ell}\, (1+|x|)^{\ell + 1}$ for $t \in [T-1/2 , T+1]$.
\end{enumerate}
\end{Thm}

   The closeness to the cylinder is given in (A), while (B) records  higher derivative bounds that will be used for interpolation.
Theorem   \ref{t:shrinkerscale}  will be proven in the last subsection.    

\begin{proof}[Proof of Theorem \ref{t:discre} assuming  Theorem \ref{t:shrinkerscale}]
Theorem \ref{t:shrinkerscale}  gives $\mu > 0$ and a compactly supported normal vector field $U$ on a (possibly different cylinder) $\cC_k'$ where $B_{(1+\mu)R_T} \cap \Sigma_T$ is contained in the graph $\Sigma_U$ of $U$.      Proposition \ref{p:gradLoj} and (A) in Theorem \ref{t:shrinkerscale} give   
 \begin{align}
 	|F(\Sigma_U) - F(\cC_k)| \leq C \, \| U \|_{L^2} \, \| \phi_U \|_{L^2} + C \, \| U \|_{L^2}^3  \leq C \,   \e^{ - \frac{ (1+\mu)   \, R_T^2}{4} }  \, .
 \end{align}
Therefore, using this and the entropy bound, we get that
 \begin{align}
 	\left| F(\Sigma_T) - F(\cC_k) \right| &\leq   |F(U)| + \int_{\Sigma_T \setminus B_{(1+\mu) \, R_T}} \e^{-f}  \leq   C \,   \e^{ - \frac{ (1+\mu)  \, R_T^2}{4} }  = C \, \left[ F(\Sigma_{T-1}) - F(\Sigma_{T+1}) \right] ^{\frac{1+\mu}{2}} \, .
		\notag
 \end{align}
 The theorem follows with $1+\alpha = \frac{2}{1+\mu}$ since $\mu > 0$.
 \end{proof}

 \subsection{Proof of Uniqueness}

 We will use Theorems \ref{t:discre} and  \ref{t:shrinkerscale} to prove uniqueness in this subsection.   Define the sequence $\delta_j \to 0$ by
\begin{align}	\label{e:deltaj}
	\delta_j = \sqrt{F (\Sigma_{j-1}) - F(\Sigma_{j+2})} =  \left( \int_{j-1}^{j+2} \| \phi \|_{L^2}^2 \right)^{ \frac{1}{2} } \, .
\end{align}
Using the Cauchy-Schwarz inequality,   $\delta_j$     bounds the $L^1$ distance from $\Sigma_j$  to $\Sigma_{j+1}$, so  uniqueness is  equivalent to  $\sum \delta_j<\infty$.  We will  show  that $\sum \delta_j^{\bar{\beta}} < \infty$   even
for some $\bar{\beta}<1$.     This
will follow from Theorem \ref{t:discre} and   the next elementary lemma:

\begin{Lem}	\label{c:djsums1}
If  there exists $\alpha \in (1/3, 1)$ and $K$ so that
\begin{align}	\label{e:assumption621}
	\left| F (\Sigma_j) - F(\cC_k) \right|^{ 1 + \alpha} \leq K \, \left[ F(\Sigma_{j-1}) - F(\Sigma_{j+1}) \right] \, , 
\end{align}
then there exists $\bar{\beta} < 1$ so that  
%\begin{align}	\label{e:djsums1}
	$\sum_{j=1}^{\infty} \, \delta_j^{ \bar{\beta}} < \infty \, .$
%\end{align}
\end{Lem}

\begin{proof} 
By \eqr{e:assumption621}, we can apply lemma $6.9$ in \cite{CM3} to get $\rho > 1$ and $C$ so that
\begin{align}	\label{e:was713}
	\sum_{i=j}^{\infty} \delta_i^2 \leq 3 \left( F(\Sigma_{i-1}) - \lim_{t\to \infty} F(\Sigma_t) \right) \leq C \, j^{ - \rho }  \, .
\end{align}
Moreover, lemma $6.9$ in \cite{CM3} shows that \eqr{e:was713}  implies that $\sum \delta_j < \infty$.  

We will show next that if $0 < q < \rho $, then
\begin{align}	\label{e:claimq}
	\sum \delta_j^2 \, j^q < \infty \, .
\end{align}
 To prove \eqr{e:claimq},  set $b_j = j^q$ and $a_j = \sum_{i=j}^{\infty} \delta_i^2$, then $a_j - a_{j+1} = \delta_j^2$ and 
\begin{align}
	  b_{j+1} - b_j =  (j+1)^q - j^q \leq c \, j^{q-1} \, ,
\end{align}
where $c$ depends on $q$ and we  used that $j\geq 1$.  Summation by parts  and \eqr{e:was713} give 
\begin{align}
	\sum_{j=k}^N \delta_j^2 j^q &= \sum_{j=k}^N b_j (a_j - a_{j+1}) = b_k a_k - b_{N} a_{N+1} + 
	\sum_{j=k}^{N-1} a_{j+1} ( b_{j+1}- b_j ) \notag \\
	&\leq k^q \, \sum_{j=k}^{\infty} \delta_j^2  + C\, \sum_{j=k}^{\infty} j^{- \rho} \, j^{q-1} \, .
\end{align}
This is bounded independently of $N$ since $q < \rho $, giving \eqr{e:claimq}.

Let $a  , \beta  $ be   constants to be chosen below.  The H\"older inequality gives    
\begin{align}	\label{e:twosums}
	\sum \delta_j^{\beta} = \sum \left( \delta_j^{\beta} \, j^a \right) \, j^{-a}  \leq \left(  \sum \delta_j^2 \, j^{ \frac{ 2a}{\beta} }   \right)^{ \frac{\beta}{2} }  \left(  \sum j^{ - \frac{2a}{2-\beta}}  \right)^{ \frac{2-\beta}{2} } \, .
\end{align}
To prove the lemma, we need $\beta < 1$ and $a>0$ so that both sums on the right in \eqr{e:twosums} are finite.  By \eqr{e:claimq}, the first is finite if $\frac{2a}{\beta}< \rho$.  The second is finite if
$  2 -  \beta < 2a$.  To satisfy both, we need
$
	2 -  \beta  < 2a < \rho \,  \beta$.  This is possible as long as $2<(1+ \rho) \, \beta$.  Since $1< \rho$, we can choose such a $\beta < 1$.
 \end{proof}
 
\begin{proof}[Proof of Theorem \ref{t:main}: Uniqueness of Blowups]
Using the  higher order bound (B),   interpolation bounds the $C^{2,\alpha}$ change from $\Sigma_j$ to $\Sigma_{j+1}$: Given $r$ satisfying (B) and any $\beta < 1$, there exists
$C_{r,\beta}$ so that the $C^{2,\alpha}$ variation of $U$ in $B_r \times [j,j+1]$ is at most
\begin{align}	\label{e:distsobad}
	C_{r,\beta} \, \delta_j^{\beta} \, .
\end{align}
By assumption, there is a sequence
  $t_i \to \infty$ so that $\Sigma_{t_i}$ converges with multiplicity one  to a cylinder. By White's Brakke estimate, \cite{W3}, this convergence is smooth on compact subsets. 
In particular, we can guarantee that
 Theorem \ref{t:discre} applies for some $t_i$ (which can be taken large).   Combining \eqr{e:distsobad} 
 and Lemma \ref{c:djsums1}, we get that $\sum_{j=1}^{\infty} \, \delta_j^{ \bar{\beta}} < \infty $ for some $\bar{\beta} < 1$.  
 Note that \eqr{e:distsobad} is used to guarantee that we can continue to apply Theorem \ref{t:discre} for $t > t_i$.
\end{proof}

\subsection{Proof of Theorem    \ref{t:shrinkerscale}}

  Theorem \ref{t:shrinkerscale}   follows  by  an iteration argument using extension and improvement.  The extension  follows from theorem $5.3$ in \cite{CM3} (since it uses Allard/Brakke and monotonicity, which apply in higher codimension).   The improvement, however, has very significant new difficulties and uses all the earlier results from this paper.
    
  We will use  the following elementary lemma in the  proof:

\begin{Lem}	\label{l:simplem}
Given $m$ and a nonnegative integer $k$, there exists $c_{m,k}$ so that  for any $R \geq 1$
\begin{align}
	\int_{\RR^m \setminus B_R} |x|^k \, \e^{ - \frac{|x|^2}{4} } \leq c_{m,k} \, R^{m+k-2} \, \e^{ - \frac{R^2}{4} }  \, .
\end{align}
\end{Lem}

\begin{proof}
Set  $   	\gamma_q (R) = \int_R^{\infty} r^q \, \e^{ - \frac{r^2}{4} } \, dr $.
Since $2\, \left( r^p \,  \e^{ - \frac{r^2}{4} } \right)' = 2p \, r^{p-1}  \,  \e^{ - \frac{r^2}{4} } - r^{p+1} \,  \,  \e^{ - \frac{r^2}{4} }$, integrating from $R$ to $\infty$ gives
\begin{align}	\label{e:gammasR}
	2 \, R^p \,  \,  \e^{ - \frac{R^2}{4} } =  -2 p \, \gamma_{p-1} (R) + \gamma_{p+1} (R) \, .
\end{align}
Applying this with $p=0$ gives $\gamma_1 (R) = 2 \,    \e^{ - \frac{R^2}{4} }$; using $p=-1$ gives $\gamma_0 (R) \leq \frac{2}{R} \,   \e^{ - \frac{R^2}{4} }$.
 Using \eqr{e:gammasR}   inductively in $p$, we get for every nonnegative integer $q$ that 
$
	\gamma_q (R) \leq c_q \, R^{q-1} \, \e^{ - \frac{R^2}{4} } \, .
$
The lemma follows from this.
\end{proof}

\begin{proof}[Proof of Theorem \ref{t:shrinkerscale}]  
The inductive argument has two components: (1) a priori estimates over the same cylinder on a larger scale and (2) an improvement over a possibly different cylinder.   

%  The one difference   is that we now also record the $L^2$ bound; this follows  from the assumed $L^2$ bound, the $L^{\infty}$ bound   and Lemma \ref{l:simplem}.

Following section $5$ in \cite{CM3} (see ($1$) in the proof of theorem $5.3$), we have{\footnote{This follows as in \cite{CM3} since the results used - White's Brakke estimates, \cite{W1},   Huisken's monotonicity, \cite{H2}, and   higher derivative bounds (lemma $3$ in \cite{AB}) - hold in all codimension.}}:
\begin{enumerate}
\item Given $\epsilon_2 > 0$, there exist $\epsilon_3 > 0$,  $R_2$, $C_2$  and $\mu > 0$  so that if  $R \in [R_2 , R_T]$ and   $B_R \cap \Sigma_t$ is  the graph of $U(x,t)$ over a fixed cylinder $\cC_k$ with $\| U \|_{C^{2,\beta}(B_R)} \leq  \epsilon_3$ for $t \in [T- \frac{1}{2}, T+1]$, then for $t \in [T- \frac{1}{2}, T+1]$:
\begin{itemize}
	\item[(A1)] $B_{(1+\mu)R} \cap \Sigma_t$ is contained in the graph of $U$ with  $\| U  \|_{C^{2,\beta}(B_{(1+\mu)R})} \leq \epsilon_2$.
	\item[(A2)] $\| \phi \|_{L^2(B_{(1+\mu)R}\cap \Sigma_t)}^2 \leq C_2 \, \e^{ - \frac{R_T^2}{2}}$.
	\item[(A3)]  $|\nabla^{\ell} A | \leq C_{\ell} \, (1 + |x|)^{\ell+1}$ on $ B_{(1+\mu)R}\cap \Sigma_t$  for each $\ell$.
	\end{itemize}
 \end{enumerate}
 \vskip2mm
 
 Observe that if $\| U \|_{L^2(B_R)}^2 \leq \e^{ - \frac{R^2}{4}}$, then Lemma \ref{l:simplem}
gives that
\begin{align}	\label{e:whatUL2has}
	\| U \|_{L^2(B_{(1+\mu)R})}^2 \leq   C_n \, \lambda_0 \, R^{n-2} \, \e^{ - \frac{R^2}{4}}  \, .
\end{align}
\vskip1mm
Choose $\epsilon_2 > 0$ so that Proposition  \ref{p:Lojae1A} and Proposition \ref{l:PLone}  apply; this fixes $\epsilon_3 > 0$, $C_2$, $\mu > 0$ and $R_2$.  The iterative hypotheses on scale $s$ will be that for $t \in [T- \frac{1}{2}, T+1]$:
\begin{itemize}
\item[($\star 1$)] $\| \phi \|_{L^2(B_s \cap \Sigma_t)}^2 \leq C_2 \, \e^{ - \frac{R_T^2}{2}}$, \, $B_s \cap \Sigma_t$ is contained in the graph of $U$ with $\| U \|_{C^{2,\beta}(B_s)} \leq \epsilon_2$,  and 
	$\| U \|_{L^2(B_s)}^2 \leq C_n \, \lambda_0 \, s^{n-2} \, \e^{ - \frac{s^2}{4(1+\mu)^2}}$.
\end{itemize}

\vskip2mm
Step (2) below will give the improvement needed to apply (1) again.  Each time that (1) is applied, the scale increases from $R$ to $s=(1+\mu)R$ but the a priori bound   gets worse.  The   bound that comes out of (1) will be good enough to apply   (2) on the scale $s$ to get improved bounds (good enough to apply (1) again) on the smaller scale $\frac{s}{1+\theta}$.   What makes it  work is that $\theta < \mu$ so that the scale keeps increasing by a definite amount.

\begin{enumerate}
 
\item[(2)]  Given $\epsilon_1 > 0$, there exists $\theta \in (0, \mu)$ and $\bar{R}$ so that if ($\star 1$) holds for some $s \in [\bar{R} , R_T]$  and all
$t \in [T- \frac{1}{2}, T+1]$, then 
there is (possibly different) cylinder so that each $B_{\frac{s}{1+\theta}} \cap \Sigma_t$ is a graph of $U$ with $\| U \|_{C^{2,\beta}} \leq \epsilon_1$ and 
\begin{align}
	\| U \|^2_{ L^2 (B_{ \frac{s}{1+\theta}})} \leq \e^{ -  \frac{ s^2}{4 \, (1+\theta)^2}} \, .
\end{align}

\end{enumerate}

\vskip1mm
The theorem follows  by applying (1) and (2) repeatedly.
It remains to prove (2).

\vskip1mm
\noindent
{\emph{Proof of (2)}}:   
 Proposition \ref{l:PLone}    gives
 $C_0$ and $\bar{\epsilon}>0$ so that if $\epsilon_2 \leq \bar{\epsilon}$, then 
 for any $\kappa \in (0,1]$ there is a constant $C_{\kappa}  $ so that (choosing $\kappa= \kappa(\alpha)$)
        \begin{align}	\label{e:e723}
        		\| P_{U} \|_{L^1} &\leq C_{\kappa} \, \| U \|_{L^2}^3  + C_{\kappa} \, \| \phi_U \|_{L^2}^{ \frac{6}{3+\kappa}}
		+ C \, \| U \|_{L^2} \,  \|  \phi_{U} \|_{L^2}  +  C \,  \|  \phi_{U} \|_{W^{1,2}}^2	\notag \\
		&\leq C_{\kappa} \, \left\{  \| U \|_{L^2}^{ \frac{6 }{3-\kappa}}    + \| \phi_U \|_{W^{1,2}}^{ \frac{6}{3+\kappa}}	\right\} \, .
\end{align}
Here the last inequality used Young's inequality on $ \| U \|_{L^2} \,  \|  \phi_{U} \|_{L^2}$.
We can now apply Theorem \ref{c:kappa} with a cutoff function $\psi$ supported 
in $B_s$ to get
\begin{align}	\label{e:ckappa2}
	\int  \psi^2 \,   |\nabla \tau|^2  \,  \e^{-f} & \leq C \, \| P_U \|_{L^1}  + C \, \int  |\nabla \psi |^2 \, \e^{-f}   + C\, \| \phi_U \|_{W^{2,1}} \notag \\
	&\qquad+ C\,  \| \nabla^{\perp} \, \phi_U \|_{L^2}^2 + C \, \| |\nabla^{\perp} \phi_U | \, |A(x^T, \cdot)| \|_{L^1}  \, .  
\end{align}
Using
the identity $\nabla_{E_i}^{\perp} \bH = 
	- \frac{1}{2} A(x^T , E_i ) - \nabla_{E_i}^{\perp} \phi $ from Lemma \ref{l:gradbH} on the last term and 
choosing a linear cutoff $\psi$ that is one on $B_{s-1}$ and using the bound  \eqr{e:e723} gives
\begin{align}	\label{e:ckappa3a}
	 \| \nabla \tau \|^2_{L^2(B_{s-1})}  \leq C_{\kappa} \, \left\{  \| U \|_{L^2}^{ \frac{6}{3-\kappa}}    + \| \phi_U \|_{W^{1,2}}^{ \frac{6}{3+\kappa}}	\right\}  + C\, \| \phi_U \|_{W^{2,1}}  + C \, s^{n-2} \, \e^{ - \frac{s^2}{4}}  \, .  
\end{align}
Next, for any $\bar{\beta} < 1$, \eqr{e:ckappa3a} and  interpolation gives that for $\| \phi_U \|_{L^2}$ sufficiently small
 \begin{align}	\label{e:ckappa3}
	 \| \nabla \tau \|^2_{L^2(B_{s-1})}  \leq C_{\kappa} \,   \| U \|_{L^2}^{ \frac{6}{3-\kappa}}    + C_{\kappa} \, \| \phi_U \|_{L^{2}}^{ \frac{6\, \bar{\beta}}{3+\kappa}}	  +   \| \phi_U \|_{L^1}^{\bar{\beta}} + C \, s^{n-2} \, \e^{ - \frac{s^2}{4}}   \, .  
\end{align}
 Since $\phi$ and $\phi_U$ only differ outside of $B_s$,   Lemma \ref{l:simplem}  and $\| \phi \|_{L^2(B_s)}^2 \leq C_2 \, \e^{ - \frac{R_T^2}{2} }$ give
\begin{align}	\label{e:phiUbd}
	\| \phi_U \|^2_{L^2} &\leq      C  \, \e^{ - \frac{R_T^2}{2} } + C \, s^{n} \, \e^{ - \frac{s^2}{4} } \leq  C \, s^{n} \, \e^{ - \frac{s^2}{4} }   \, , \\
	\| \phi_U \|_{L^1} &\leq   C  \, \e^{ - \frac{R_T^2}{4} } + C \, s^{n-1} \, \e^{ - \frac{s^2}{4} }\leq  C \, s^{n-1} \, \e^{ - \frac{s^2}{4} } 
	\, .   \label{e:phiUbdL1}
\end{align}
Using \eqr{e:phiUbd} and \eqr{e:phiUbdL1} in \eqr{e:ckappa3}  gives
  \begin{align}	\label{e:ckappa4}
	 \| \nabla \tau \|^2_{L^2(B_{s-1})}  &\leq C_{\kappa} \,   \| U \|_{L^2}^{ \frac{6}{3-\kappa}}    + C_{\kappa} \,   s^{n} \, \e^{ - \left( \frac{3\bar{\beta}}{3+\kappa} \right) \, \frac{s^2}{4}} \, . 
\end{align}

Given any $\beta < 1$,
applying interpolation to \eqr{e:phiUbdL1} and \eqr{e:ckappa4} gives $C_{\beta,\kappa}$ so that on $B_{s-2}$
\begin{align}	\label{e:phiboundsnow}
	\left( |\nabla \tau|^2 + |\nabla^2 \tau|^2 +  |\phi| + |\nabla \phi|  + |\nabla^2 \phi |   \right) \, \e^{ \frac{-|x|^2}{4} } \leq C_{\beta,\kappa} \,  \e^{ - \left( \frac{3\bar{\beta}}{3+\kappa} \right)\frac{\beta\,s^2}{4} }
	+  C_{\beta,\kappa} \,  \, \| U \|_{L^2}^{ \frac{6 \,\beta}{3-\kappa}} \,.
\end{align}
Using  $\| U \|_{L^2(B_s)}^2 \leq C_n \, \lambda_0 \, s^{n-2} \, \e^{ - \frac{s^2}{4(1+\mu)^2}}$ from ($\star 1$), we can bound \eqr{e:phiboundsnow} by
\begin{align}	\label{e:phiboundsnow2}
	\leq C_{\beta,\kappa} \,  \e^{ - \left( \frac{3\bar{\beta}}{3+\kappa} \right)\frac{\beta\,s^2}{4} }
	+  C_{\beta,\kappa} \,  \e^{ - \frac{3\, \beta}{3-\kappa} \, \frac{s^2}{4(1+\mu)^2} } \, .
\end{align}
Now, we choose $\kappa$ so that $(1+ \mu)^2 = \frac{3+\kappa }{3-\kappa}$ and, thus, 
\begin{align}
	\frac{3}{(3-\kappa)(1+\mu)^2} = \frac{3}{3+\kappa}  > \frac{3-\kappa}{3+\kappa} = \frac{1}{(1+\mu)^2} \, .
\end{align}
This gives some $\theta < \mu$ so that we can apply Proposition  \ref{p:Lojae1A} on the scale $\frac{s}{1+\theta}$.  In particular, 
 as long as $\bar{R}$ is large enough, 
 Proposition  \ref{p:Lojae1A}   shows that each $\Sigma_t$ is the graph over a cylinder, which may depend on $t$, satisfying the desired  bounds.  The bounds on 
  $\phi$ in \eqr{e:phiboundsnow} control the change in  the flow over time, so the cylinder does not depend on $t$.
\end{proof}

\end{document}